\documentclass[reqno,11pt]{amsart}

\usepackage[utf8]{inputenc} % For inputting international characters; this should be on by default after 2018
\usepackage[T1]{fontenc} % Output font encoding for international characters

\usepackage[normalem]{ulem}

\usepackage{fullpage}
\usepackage{mathtools}
\usepackage{amssymb}
\usepackage{dsfont}
\usepackage{subcaption}

\usepackage{tikz}
\usetikzlibrary{decorations.markings, decorations.text}
\usetikzlibrary{cd, shapes}

\tikzset{->-/.default=0.5, ->-/.style={decoration={
  markings,
  mark=at position #1 with {\arrow{stealth}}},postaction={decorate}}}

\usepackage{hyperref}

\usepackage[nobysame,alphabetic,initials,msc-links]{amsrefs}

\DefineSimpleKey{bib}{how}

\renewcommand{\eprint}[1]{#1}
\BibSpec{misc}{%
  +{}{\PrintAuthors}  {author}
  +{,}{ \textit}      {title}
  +{,}{ }             {how}
  +{}{ \parenthesize} {date}
  +{,} { available at \eprint}        {eprint}
  +{,}{ available at \url}{url}
  +{,}{ }             {note}
  +{.}{}              {transition}
}

\numberwithin{equation}{section}

\theoremstyle{plain}%default
\newtheorem{theorem}{Theorem}[section]
\newtheorem{proposition}[theorem]{Proposition}
\newtheorem{lemma}[theorem]{Lemma}
\newtheorem{corollary}[theorem]{Corollary}

\theoremstyle{definition}
\newtheorem{definition}[theorem]{Definition}

\theoremstyle{remark}
\newtheorem{remark}[theorem]{Remark}
\newtheorem{example}[theorem]{Example}

\newcommand\bp{\begin{proof}}
\newcommand\ep{\end{proof}}

\newcommand\ee{\nopagebreak\mbox{\ }\hfill$\diamond$}

\newcommand{\un}{\mathds{1}}

\newcommand\C{\mathbb{C}}
\newcommand\T{\mathbb{T}}
\newcommand\R{\mathbb{R}}
\newcommand\Q{\mathbb{Q}}
\newcommand\Z{\mathbb{Z}}

\newcommand{\A}{\mathcal{A}}
\newcommand{\G}{\mathcal{G}}
\newcommand{\F}{\mathcal{F}}
\newcommand{\K}{\mathcal{K}}
\newcommand{\LL}{\mathcal{L}}
\newcommand{\HH}{\mathcal{H}}
\newcommand{\OO}{\mathcal{O}}
\newcommand{\TT}{\mathcal{T}}

\newcommand{\WW}{\mathcal{W}}

\newcommand{\cG}{\mathcal G}

\newcommand{\Gu}{{\mathcal{G}^{(0)}}}

\newcommand\Ad{\operatorname{Ad}}

\newcommand\CAR{\operatorname{CAR}}
\newcommand\id{\operatorname{id}}
\newcommand\Ind{\operatorname{Ind}}
\newcommand\Mat{\operatorname{Mat}}
\newcommand\supp{\operatorname{supp}}
\newcommand\Tr{\operatorname{Tr}}

\newcommand{\KK}{\mathrm{KK}}

\newcommand\eps{\varepsilon}

\newcommand{\stub}{\makebox[1.8ex]{$\cdot$}}

\begin{document}

\date{February 14, 2024; revised December 18, 2024}

\title{Crystallization of C$^*$-algebras}

\author[Laca]{Marcelo Laca}
\address{Department of Mathematics and Statistics, University of
Victoria, P.O. Box 1700 STN CSC, Victoria, British Columbia, V8W 2Y2, Canada.}
\email{laca@uvic.ca}

\author[Neshveyev]{Sergey Neshveyev}
\address{Department of Mathematics, University of Oslo,
P.O. Box 1053 Blindern, NO-0316 Oslo, Norway}
\email{sergeyn@math.uio.no}
\thanks{Supported by the NFR project 300837 ``Quantum Symmetry''}

\author[Yamashita]{Makoto Yamashita}
\address{Department of Mathematics, University of Oslo,
P.O. Box 1053 Blindern, NO-0316 Oslo, Norway}
\email{makotoy@math.uio.no}

\begin{abstract}
Given a C$^*$-algebra $A$ with an almost periodic time evolution $\sigma$, we define a new C$^*$-algebra $A_c$, which we call the crystal of~$(A,\sigma)$, that
represents the zero temperature limit of~$(A, \sigma)$. We prove that there is a one-to-one correspondence between the ground states of $(A,\sigma)$ and the states on $A_c$, justifying the name. In order to  investigate further the relation between low temperature equilibrium states on $A$ and traces on $A_c$,
we define a Fock module $\F$ over the crystal  and construct a vacuum representation of $A$ on $\F$. This allows us to show, under relatively mild assumptions,  that  for sufficiently large inverse temperatures $\beta$ the $\sigma$-KMS$_\beta$-states on $A$ are induced from traces on $A_c$ by means of the Fock module. In the second part, we compare the K-theoretic structures of $A$ and $A_c$. Previous work by various authors suggests that they have (rationally) isomorphic K-groups. We analyze this phenomenon in detail, confirming it under favorable conditions, but showing that, in general, there is apparently no easy way to relate these groups. As examples, we discuss in particular Exel's results on semi-saturated circle actions, and recent results of Miller on the K-theory of inverse semigroup C$^*$-algebras. In relation to the latter, we introduce the notion of a scale~$N$ on an inverse semigroup~$I$ and define a new inverse semigroup~$I_c$, which we call the crystal of~$(I,N)$.
\end{abstract}

\maketitle

\section{Introduction}

It has been observed in a number of cases that for large inverse temperatures $\beta$ the description of KMS$_\beta$-states on a C$^*$-algebra~$A$ with a time evolution $\sigma$ is related to the description of the K-theory of~$A$. One prominent example of this is given by the Pimsner--Toeplitz algebras $\TT_X$ equipped with a (generalized) gauge action. By a well-known result of Pimsner~\cite{MR1426840}, the C$^*$-algebra $\TT_X$ is KK-equivalent to the coefficient algebra of the C$^*$-correspondence $X$. On the other hand, as was shown by Laca and Neshveyev~\cite{MR2056837}, the KMS$_\beta$-states on $\TT_X$ are described in terms of traces on the coefficient algebra. Another class of examples is provided by the semigroup (or Toeplitz) C$^*$-algebras of $ax+b$ semigroups. The study of their KMS-states with respect to a natural time evolution was initiated by Laca and Raeburn~\cite{MR2671177}, see~\cites{MR2785897,MR3037019,N,MR4227582} for subsequent developments. A quite general formula for the K-theory of such C$^*$-algebras was then obtained by Cuntz, Echterhoff and Li~\cite{MR3323201}. Again one sees that both the KMS-states and K-theory are described in terms of traces and K-theory of another C$^*$-algebra with a more transparent structure.

In 2015, during the discussion of this phenomenon at the Abel Symposium ``Operator Algebras and Applications'', Connes asked whether it is possible to ``cool down'' a thermodynamical system~$(A,\sigma)$ to obtain a simpler  C$^*$-algebra that is more ``classical'' (perhaps commutative or at least type I), but still retains  significant topological information about $A$. A similar idea had been previously realized by Connes, Consani and Marcolli in their work on endomotives~\cite{MR2349719}.

The idea of cooling down, or crystallization, as we prefer to call it, has  been used in mathematics before. Probably the most famous and successful example is the theory of crystal bases of Kashiwara~\cite{MR1090425} and Lusztig~\cite{MR1035415}. In this theory, one starts with a semisimple Lie algebra~$\mathfrak g$, takes the Drinfeld--Jimbo $q$-deformation $U_q\mathfrak g$ of its universal enveloping algebra and then lets the deformation parameter $q$ tend to $0$. It turns out that in the limit the action of the standard generators of~$\mathfrak g$  on the limit $\mathfrak g$-modules takes a much simpler form  than for $\mathfrak g$, and this allows for a combinatorial analysis of %\tcr{yet they still contain valuable information about}
the representation theory of $\mathfrak g$.

On the dual side of quantized function algebras, the crystallization phenomenon has  also been observed for function algebras of quantum groups and homogeneous spaces.  This first appeared in the work of Woronowicz \cite{MR890482}, who introduced the C$^*$-algebras $C(\mathit{SU}_q(2))$ of ``continuous functions on the quantum group $\mathit{SU}_q(2)$'' for a quantization parameter $-1 \le q \le 1$, and showed that for $-1 < q < 1$ these C$^*$-algebras are all isomorphic to the case $q = 0$, which admits a more combinatorial description.
This idea has been further developed by Hong and Szyma\'{n}ski in \cite{MR1942860}, where they show that the $q$-deformed algebras of functions on certain homogeneous spaces  are graph C$^*$-algebras. One of the key points in their work is that the presentation of graph C$^*$-algebras is easy to recognize  by letting $q=0$ in the relations describing the quantized function algebras.

More recently, Giselsson~\cite{Gis} applied the same idea to show that $C(SU_q(3))$ is the C$^*$-algebra of a $2$-graph. Soon afterwards, Matassa and Yuncken~\cite{MR4635346} established an explicit connection of these ideas with the work Kashiwara and Lusztig. Using the theory of crystal bases, they described a limit of the quantized function algebras $C(G_q)$ when $q\to0$ as a higher rank graph C$^*$-algebra for all simply connected compact semisimple Lie groups $G$. A limit of the algebras $C(SU_q(n))$ when $q\to0$ has also been defined by Giri and Pal~\cite{arXiv:2203.14665}.

In this work, motivated by  the question of Connes,  and guided by the above mentioned examples and results, we start with an almost periodic time evolution $\sigma$ on a  C$^*$-algebra $A$  and  define a new C$^*$-algebra $A_c$, which we call the \emph{crystal} of $(A,\sigma)$ (Definition~\ref{def:crystal}). It is an almost immediate consequence of the construction that the states on $A_c$ are in a one-to-one correspondence with the ground (or temperature zero) states on $A$ (Proposition~\ref{prop:ground}). The construction passes several other simple tests. Namely, for every Pimsner--Toeplitz algebra~$\TT_X$ with the gauge action it returns the coefficient algebra of $X$ (see Example~\ref{ex:TP}), and for groupoid C$^*$-algebras it agrees with the construction in~\cite{MR3892090} (Proposition~\ref{prop:groupoid-crystal}).

The goal of the first part of the paper (Section~\ref{sec:KMS}) is to understand when $\sigma$-KMS$_\beta$-states on~$A$ can be described in terms of traces on $A_c$. The analysis here is motivated by the theory developed in~\cite{MR2056837} and its generalization in~\cite{MR4150892}. What allows us to use techniques from those papers is that although $A$ is not ostensibly a quantized algebra to begin with, the presence of  an almost periodic dynamics on $A$  implies that a quotient of $A$ is a type of Fock quantization of the crystal. Specifically, we construct a right C$^*$-Hilbert $A_c$-module $\F$, which we call the \emph{Fock module}, and a representation $\Lambda\colon A\to\LL(\F)$, thus showing that $\Lambda(A)$ has structure reminiscent of the Pimsner--Toeplitz algebra of a product system of C$^*$-correspondences over $A_c$. In the particular case of a genuine Pimsner--Toeplitz algebra endowed with the gauge action, the crystal coincides with the coefficient algebra and this recovers the usual construction.  For a general almost periodic dynamics  we can use our Fock module to induce traces on $A_c$ to KMS$_\beta$-states on $A$. Our main result here (Theorem~\ref{thm:KMS-full}) shows that under quite general assumptions, this way we get all KMS$_\beta$-states on $A$ for sufficiently large $\beta$.

The approach to low-temperature KMS-states through the crystal is very efficient; one can even say  surprisingly so. In a number of previously known cases it gives the quickest way, requiring almost no guess work, to a complete classification of KMS$_\beta$-states and to explicit formulas for such states. We illustrate this with C$^*$-algebras of LCM semigroups $S$ equipped with the dynamics defined by a homomorphism $N\colon S\to [1,+\infty)$ (Example~\ref{ex:LCM}). It should be stressed again that our results are for large $\beta$. What happens for small $\beta$ and when a phase transition occurs, seem to be difficult to determine in our general setting.

In the second part of the paper we compare the K-theories of $A$ and $A_c$. Because of the analogy with Pimsner--Toeplitz algebras, one might hope that there is a direct connection between them, but in Section~\ref{ssec:funct} we give several examples showing that there are no natural transformations between the functors $(A,\sigma)\to K_*(A)$ and $(A,\sigma)\to K_*(A_c)$ satisfying relatively mild and seemingly reasonable assumptions. The connection between these functors requires therefore further investigation. In this paper we will be content with giving several classes of examples where $K_*(A)$ and $K_*(A_c)$ are (rationally) isomorphic.
%So all we can say for now  is that it is reasonable to assume that $A_c$ contains some information about the K-theory of $A$, but how much and how exactly depends %on the concrete situation.

In Section~\ref{ssec:semis} we consider semi-saturated circle actions in the sense of Exel~\cite{MR1276163} and show that under certain conditions we have $K_*(A)\otimes_\Z\Q\cong K_*(A_c)\otimes_\Z\Q$ (Theorem~\ref{thm:K-theory-circle}). Although the conditions are arguably reasonable, it is not difficult to violate them. In particular, in Example~\ref{ex:dynam} we show how to construct a C$^*$-algebra $A$ with a periodic dynamics such that $A_c=\C$, the representation $\Lambda\colon A\to\LL(\F)$ is faithful, but $K_0(A)\cong K_0(C(K)\rtimes_S\Z)$ for a given minimal homeomorphism~$S$ of the Cantor set~$K$, so in this case the crystal contains no interesting information about the K-theory of~$A$.

In Section~\ref{sec:inverse} we study the C$^*$-algebras of inverse semigroups $I$ equipped with the dynamics defined by a \emph{scale} $N\colon I^\times\to\R^\times_+$. We first show that the crystal in this case is again the C$^*$-algebra of an inverse semigroup $I_c$, which we call the crystal of $(I,N)$ (Corollary~\ref{cor:inverse-crystal}). This puts us into position to apply recent results of Miller~\cite{miller-thesis} on K-theory of C$^*$-algebras of inverse semigroups and obtain sufficient conditions for an isomorphism $K_*(A)\cong K_*(A_c)$ (Theorem~\ref{thm:k-theory-isom-inv-semigrp}).

\smallskip

{\bf Acknowledgement.} It is our pleasure to thank Chris Bruce and Xin Li for discussions at the early stage of the project.

\bigskip

\section{Crystallization}

\subsection{Crystals and ground states}\label{sec:crystal-c-star-algs}

Assume $A$ is a C$^*$-algebra and $\sigma=(\sigma_t)_{t\in\R}$ is a one-parameter group of automorphisms of $A$. For $\lambda\in\R$, define
\[
A_\lambda=\{a\in A\colon \text{$\sigma_t(a)=e^{it\lambda}a$ for all $t\in\R$}\}.
\]
We assume that the dynamics $\sigma$ is almost periodic, meaning that the $*$-subalgebra $\A$ of $A$ spanned by the spaces $A_\lambda$, $\lambda\in\R$, is dense in $A$. Denote by $I_\lambda$ the ideal $\overline{A_\lambda A_\lambda^*}$ in $A_0$, and let
\[
I=\overline{\sum_{\lambda>0}I_\lambda}.
\]

\begin{definition}\label{def:crystal}
We call the C$^*$-algebra $A_c=A_0/I$ the \emph{crystal} of $(A,\sigma)$.
\end{definition}

Observe that the construction of the crystal is functorial, meaning that given two dynamical systems $(A_1,\sigma_1)$ and $(A_2,\sigma_2)$ and an equivariant $*$-homomorphism $\pi\colon A_1\to A_2$, we get a $*$-homomorphism $\pi_*\colon A_{1,c}\to A_{2,c}$.

To justify the name, recall that a ground state, or a zero temperature equilibrium state, is a state~$\phi$ on~$A$ such that for every $a\in A$ and every $\sigma$-analytic element $b\in A$, the function $z\mapsto \phi(a\sigma_z(b))$ is bounded in the region $\operatorname{Im} z\ge0$, see, e.g., \cite{BR2}*{Section~5.3}. Such a state is automatically $\sigma$-invariant. Furthermore, it suffices to check boundedness for all $a$ and $b$ in a set of analytic elements that spans a dense subspace of $A$. For later use, recall also that a $\sigma$-invariant state $\phi$ is called a $\sigma$-KMS$_\beta$-state ($\beta\in\R$), if $\phi(ba)=\phi(a\sigma_{i\beta}(b))$ for all $a\in A$ and all $\sigma$-analytic elements $b\in A$.

Let $\Gamma\subset\R$ be the additive subgroup generated by the numbers $\lambda\in\R$ such that $A_\lambda\ne0$. We view~$\Gamma$ as a discrete group. Consider the homomorphism $\R\to\hat\Gamma$ dual to the embedding $\Gamma\to\R$, where we identify $\hat\R$ with $\R$ using the pairing $\R\times\R\to\T$, $(x,y)\mapsto e^{ixy}$. The image of $\R$ in $\hat\Gamma$ is dense and the map $\R\ni t\mapsto\sigma_t(a)\in A$ extends by continuity to $\hat\Gamma$ for all $a\in\A$, hence for all $a\in A$ by the density of $\A$ in $A$. We thus get an action of $\hat\Gamma$ on $A$, which we continue to denote by~$\sigma$, and a conditional expectation
\begin{equation}\label{eq:cond-exp}
E\colon A\to A_0,\quad E(a)=\int_{\hat\Gamma}\sigma_\chi(a)\,d\chi,
\end{equation}
such  that $E(A_\lambda)=0$ for $\lambda\ne0$.

\begin{proposition}\label{prop:ground}
 Let $\sigma$ be an almost periodic dynamics on $A$. Then there is a completely positive contraction onto the crystal obtained by taking $\hat{\Gamma}$-averages modulo $I$,
\[
\vartheta\colon A\to A_c=A_0/I,\qquad \vartheta(a)=E(a)+I.
\]
If $A_c\ne0$, then the map $\psi\mapsto\psi\circ\vartheta$ defines a bijection between the state space of~$A_c$ and the $\sigma$-ground states on $A$. If $A_c=0$, then there are no $\sigma$-ground states on $A$.
\end{proposition}

\bp
The first assertion is clearly true. Since our dynamics is almost periodic, the subspaces $A_\lambda$ with $\lambda \in \R$ span a dense subalgebra of $\sigma$-analytic elements, so the ground states are the states $\phi$ for which the function $z\mapsto\phi(b\sigma_z(a))=e^{i\lambda z}\phi(ba)$ is bounded on the upper half-plane for all $b\in\A$, $a\in A_\lambda$ and $\lambda\in\R$, which happens if and only if $\phi(ba)=0$ for all $b\in \A$ and $a\in A_\lambda$ with $\lambda<0$. By the Cauchy--Schwarz inequality this is equivalent to $\phi(a^*a)=0$ for all $a\in A_\lambda$ with $\lambda<0$. Since $A_\lambda^*=A_{-\lambda}$, this is in turn equivalent to saying that $\phi$ factors through $A_0$ and vanishes on $\overline{A_\lambda^*A_\lambda}=I_{-\lambda}$  for all $\lambda<0$. This gives the result.
\ep

The following is one of the main motivating examples for our construction.

\begin{example}\label{ex:TP}
Let $B$ be a C$^*$-algebra and $Y$ is a C$^*$-correspondence over $B$, that is, a right C$^*$-Hilbert $B$-module together with a left action of $B$ by adjointable operators.
Consider the Pimsner--Toeplitz algebra $A=\TT_Y$~\cite{MR1426840}, so $\TT_Y$ is generated by $B$ and elements $S_\xi$, $\xi\in Y$, such that the map $Y\to\TT_Y$, $\xi\mapsto S_\xi$, is $B$-bilinear and $S^*_\xi S_\zeta=\langle\xi,\zeta\rangle_B$ for all $\xi,\zeta\in Y$.

Consider a Hamiltonian $H$ on $Y$ with strictly positive and pure point spectrum, that is, we assume that $Y=\bigoplus_{\lambda>0}Y_\lambda$ for C$^*$-correspondences $Y_\lambda$ such that $H\xi=\lambda\xi$ for all $\xi\in Y_\lambda$. Then $H$ defines a one-parameter group of automorphisms of $\TT_Y$ by
\[
\sigma^H_t(b)=b\quad\text{for}\quad b\in B,\qquad \sigma^H_t(S_\xi)=e^{i\lambda t}S_\xi\quad\text{for}\quad \xi\in Y_\lambda.
\]
When $Y_\lambda=0$ for all $\lambda\ne1$, we get the gauge action $\gamma$.

From the relations in $\TT_Y$ we easily get that $B$ together with the elements of the form
\begin{equation}\label{eq:T-generators}
S_{\xi_1}\dots S_{\xi_n}S^*_{\zeta_1}\dots S^*_{\zeta_m},
\end{equation}
where $\xi_i\in Y_{\lambda_i}$ and $\zeta_j\in Y_{\mu_j}$,  $n,m\ge0$, $n+m\ge1$, span a dense subspace of $\TT_Y$. It follows that~$B$ together with the elements~\eqref{eq:T-generators} such that
$\lambda_1+\dots+\lambda_n=\mu_1+\dots+\mu_m$ span a dense subspace of~$A_0$. Similarly, the space $A_\lambda$ for $\lambda>0$ is the closure of the linear span of elements~\eqref{eq:T-generators} such that
\[
\lambda=\lambda_1+\dots+\lambda_n-\mu_1-\dots-\mu_m.
\]
Note that for the last equality to hold we must have $n\ge1$. From this we can conclude that $I=\overline{\sum_{\lambda>0}A_\lambda A^*_\lambda}$ is the closure of the linear span of elements~\eqref{eq:T-generators} such that $\lambda_1+\dots+\lambda_n=\mu_1+\dots+\mu_m$. Therefore the map $B\to A_c=A_0/I$ is surjective.

In order to see that this map is injective, consider the Fock module
\[
\F_Y=B\oplus Y\oplus (Y\otimes_B Y)\oplus\dots
\]
and the standard representation of $\TT_Y$ on it given by
\[
S_\xi(\xi_1\otimes\dots\otimes\xi_n)=\xi\otimes\xi_1\otimes\dots\otimes\xi_n.
\]
Consider the projection $Q_0\colon \F_Y\to B$ onto the $0$th component. Then $Q_0xQ_0=0$ for all $x\in I$ and $Q_0bQ_0=b$ for all $b\in B$, so the compression by $Q_0$ defines a contractive map $A_c\to B$ that is a left inverse to $B\to A_c$.

To summarize, the crystal of $(\TT_Y,\sigma^H)$ is the coefficient algebra $B$. Proposition~\ref{prop:ground} in this case is a consequence of \cite{MR2056837}*{Theorem~2.2}, because the positive energy condition needed there is satisfied when the Hamiltonian has strictly positive pure point spectrum.

We note in passing that if instead of $\TT_Y$ we consider the Cuntz--Pimsner algebra $\OO_Y$, then the crystal of $(\OO_Y,\sigma^H)$ is zero whenever the left action of $B$ on $Y$ is by generalized compact operators. This includes the systems studied in~\cite{MR1785460}.
\end{example}

\subsection{Crystals of groupoid \texorpdfstring{C$^*$}{C*}-algebras}\label{sec:crystal-gropd}

%In the context of groupoid algebras the algebra $A_c$ has been already introduced in~\cite{MR3892090}.

Let $\G$ be a locally compact, not necessarily Hausdorff, étale groupoid. More precisely, we assume that $\G$ is a topological groupoid with locally compact underlying space such that the unit space $\Gu$ is a locally compact Hausdorff space in the relative topology and the source $s\colon\G\to\Gu$ and range $r\colon\G\to\Gu$ maps are local homeomorphisms. Denote by $C_c(\G)$ the space spanned by the (not necessarily continuous) functions $f$ on $\G$ such that~$f$ is zero outside an open Hausdorff subset $W\subset \G$ and the restriction of $f$ to $W$ is continuous and compactly supported. This is a $*$-algebra with convolution product
\begin{equation*} \label{eprod}
(f_{1}*f_{2})(g) = \sum_{h \in \G^{r(g)}} f_{1}(h) f_{2}(h^{-1}g)
\end{equation*}
and involution by $f^{*}(g)=\overline{f(g^{-1})}$. Denote by $C^*(\G)$ the C$^*$-enveloping algebra of this $*$-algebra.

Assume $c$ is a locally constant real-valued $1$-cocycle on $\G$, that is, a continuous homomorphism $c\colon\G\to\R_d$, where $\R_d$ denotes the additive group~$\R$ with the discrete topology. We then have an almost periodic dynamics $\sigma^c$ on $C^*(\G)$ defined by
\[
\sigma^c_t(f)(g)=e^{itc(g)}f(g)\quad\text{for}\quad f\in C_c(\G)\quad \text{and}\quad g\in\G.
\]

Recall from  \citelist{\cite{MR584266}\cite{MR3892090}} that the \emph{boundary set} of $c$ is defined by
\begin{equation}\label{eq:boundary}
Z=\{x\in\Gu\colon \text{$c\ge0$ on $\G_x$}\}=\{x\in\Gu\colon \text{$c\le0$ on $\G^x$}\}.
\end{equation}
Denote by $\G|_Z$ the reduction of $\G$ by $Z$.
If $Z\ne\emptyset$, this is a locally compact étale groupoid, and by \cite{MR3892090}*{Theorem~1.4} there is a bijection between the state space of $C^*(\G|_Z)$ and the $\sigma^c$-ground states on $C^*(\G)$, at least in the second countable Hausdorff case.
We have the following result, which in particular shows that these extra assumptions on $\G$ are not needed; \cite{MR3892090}*{Theorem~1.4} allows for slightly more general cocycles $c$ though.

\begin{proposition}\label{prop:groupoid-crystal}
Assume $\G$ is a locally compact étale groupoid and $c\colon\G\to\R$ is a locally constant $1$-cocycle with boundary set $Z$. Consider the C$^*$-algebra $A=C^*(\G)$ and the dynamics $\sigma^c$. Then the crystal $A_c$ of $(C^*(\G),\sigma^c)$ is canonically isomorphic to $C^*(\G|_Z)$. In particular, $A_c\ne0$ if and only if $Z\ne\emptyset$.
\end{proposition}

\bp
Let us show first that the closure of $C_c(c^{-1}(0))$ in $C^*(\G)$ can be identified with $C^*(c^{-1}(0))$. This can be shown along the lines of the proof of \cite{MR3892090}*{Theorem~1.4}, but since formally that proof relies on second countability, let us give an independent argument.

Note that the conditional expectation $E\colon A\to A_0$ defined by~\eqref{eq:cond-exp} has the property $E(f)=f|_{c^{-1}(0)}$ for $f\in C_c(\G)$, where we view $C_c(c^{-1}(0))$ as a subspace of $C_c(\G)$ by extending functions by zero. It follows that $A_0$ is the closure of $C_c(c^{-1}(0))$ in $C^*(\G)$. We claim that it therefore suffices to show that for any $f\in C_c(\G)$ we can write $E(f^**f)$ as a finite sum of functions $h^**h$ with $h\in C_c(c^{-1}(0))$. Indeed, we can then define a $C_c(c^{-1}(0))$-valued inner product
\[
\langle \stub,\stub \rangle\colon C_c(\G)\times C_c(\G)\to C_c(c^{-1}(0))\quad\text{by}\quad \langle f_1,f_2\rangle=E(f_1^**f_2)
\]
and then induce any representation $\rho$ of $C_c(c^{-1}(0))$ on a Hilbert space to a representation $\pi$ of $C_c(\G)$. Since $\pi|_{C_c(c^{-1}(0))}$ contains $\rho$ as a subrepresentation, it follows that the completion of $C_c(c^{-1}(0))$ with respect to the norm on $C^*(\G)$ coincides with $C^*(c^{-1}(0))$.

Since every $f\in C_c(\G)$ is zero outside the union of finitely many sets $c^{-1}(\lambda)$, it is enough to consider $f\in C_c(c^{-1}(\lambda))$ for $\lambda\ne0$. We can find open bisections $U_1,\dots, U_n\subset c^{-1}(\lambda)$ and compact subsets $K_i\subset U_i$ such that $f$ is zero outside $\bigcup^{n}_{i=1} K_i$. We can also find functions $\rho_i\in C_c(\Gu)$ such that $\supp\rho_i\subset r(U_i)$, $0\le\rho_i\le1$ and $\sum^n_{i=1}\rho_i(x)=1$ for all $x\in\bigcup^{n}_{i=1}r(K_i)$. Define functions $f_i\in C_c(U_i)$ by letting $f_i(g)=\rho_i(r(g))^{1/2}$ for $g\in U_i$. We then have
\[
f=\sum_i f_i*f_i^**f=\sum_if_i*\langle f_i,f\rangle,
\]
hence
\[
E(f^* * f) = \langle f,f\rangle=\sum_i\langle f,f_i*\langle f_i,f\rangle\rangle=\sum_i\langle f_i,f\rangle^**\langle f_i,f\rangle,
\]
which gives the required decomposition.

\smallskip

Next, recall that the boundary set $Z$ is invariant under the action of $c^{-1}(0)$ and we have $\G|_Z=c^{-1}(0)|_Z$, see~\cite{MR584266}*{Proposition I.3.16}.
Consider the kernel of the restriction map $C^*(c^{-1}(0))\to C^*(c^{-1}(0)|_Z)$. By \cite{MR4563262}*{Proposition~1.1}, it contains $C_c(c^{-1}(0)\setminus c^{-1}(0)|_Z)$ as a dense subspace. It follows that in order to finish the proof of the proposition it suffices to show that the algebraic ideal $J\subset C_c(c^{-1}(0))$ spanned by the functions $f_1*f_2^*$ for all $f_1,f_2\in C_c(c^{-1}(\lambda))$ and $\lambda>0$ coincides with $C_c(c^{-1}(0)\setminus c^{-1}(0)|_Z)$.

The inclusion $J\subset C_c(c^{-1}(0)\setminus c^{-1}(0)|_Z)$ is clear, since if $g\in c^{-1}(\lambda)$ for some $\lambda>0$, then $r(g)\notin Z$. For the opposite inclusion it suffices to show that for every $g\in c^{-1}(0)\setminus c^{-1}(0)|_Z$ there is an open Hausdorff neighbourhood $U$ of $g$ such that every function $f\in C_c(U)$ can be written as $f_1*f_2^*$ for some $f_1,f_2\in C_c(c^{-1}(\lambda))$ and $\lambda>0$. As $r(g)\notin Z$, there is $h\in \G^{r(g)}$ such that $\lambda=c(h)>0$. Choose an open bisection $V\subset c^{-1}(\lambda)$ containing~$h$. Then choose an open bisection $U\subset c^{-1}(0)$ such that $g\in U$ and $r(U)\subset r(V)$. Consider also the bisection $U^{-1} V\subset c^{-1}(\lambda)$ containing $g^{-1}h$. Now, given $f\in C_c(U)$, define $f_1\in C_c(V)$ by letting $f_1(g')=f(g'')$ for $g'\in V\cap r^{-1}(U)$, where $g''\in U$ is the unique point with $r(g'')=r(g')$, and $f(g')=0$ for $g'\in V\setminus r^{-1}(U)$. Take also any $f_2\in C_c(U^{-1} V)$ such that $f_2(g')=1$ whenever there is (a necessarily unique) $g''\in U$ such that $f(g'')\ne0$ and $s(g'')=r(g')$. Then $f=f_1*f_2^*$.
\ep

\begin{remark}\label{rem:reduced}
More generally, instead of the full groupoid C$^*$-algebra $C^*(\G)$ we can consider the completion $C^*_\nu(\G)$ of $C_c(\G)$ with respect to some $\sigma^c$-invariant C$^*$-norm $\|\stub\|_\nu$ dominating the reduced norm. By restriction we then get a norm on  $C_c(c^{-1}(0))$. Then by \cite{CN2}*{Proposition~1.2} we also get a norm on $C_c(\G|_Z)$ such that the sequence
\[
0\to C^*_\nu(c^{-1}(0)\setminus c^{-1}(0)|_Z)\to C^*_\nu(c^{-1}(0))\to C^*_\nu(\G|_Z)\to0
\]
is exact. The same arguments as in the above proof show that the crystal of $(C^*_\nu(\G),\sigma^c)$ is~$C^*_\nu(\G|_Z)$.

Note that if we start with the reduced norm on $C_c(\G)$, then the crystal $C^*_\nu(\G|_Z)$ might be different from $C^*_r(\G|_Z)$, but we do not have an example where this actually happens, cf.~\cite{CN1}.
\end{remark}

\bigskip

\section{Equilibrium states defined by crystals} \label{sec:KMS}

\subsection{Fock module and vacuum representations}
We continue to assume that $A$ is a C$^*$-algebra with an almost periodic time evolution $\sigma$. For every $\lambda\in\R$, we can view the spectral subspace $A_\lambda$ as a right C$^*$-Hilbert $A_0$-module with inner product $\langle a,b\rangle=a^*b$. Consider the closed submodule $\overline{A_\lambda I}\subset A_\lambda$, where, as before, $I=\overline{\sum_{\lambda>0}A_\lambda A^*_\lambda}$. If $(e_i)_i$ is an approximate unit in $I$, then $a\in A_\lambda$ lies in $\overline{A_\lambda I}$ if and only if $a e_i\to a$. It follows that the quotient norm on the Banach space
\[
X_\lambda=A_\lambda/\overline{A_\lambda I}
\]
can be computed as $\|a+\overline{A_\lambda I}\|=\lim_i\|a-ae_i\|$. We then conclude that $X_\lambda$ is a right C$^*$-Hilbert $A_c$-module, with the quotient norm equal to the one defined by the inner product
\[
\langle a+\overline{A_\lambda I},b+\overline{A_\lambda I}\rangle=a^*b+I.
\]
We can also write $X_\lambda=A_\lambda\otimes_{A_0}(A_0/I)$, the main point of the above presentation is to stress that~$X_\lambda$ is a quotient of $A_\lambda$.

By the definition of $I$ we have $X_0=A_c$ and $X_\lambda=0$ for all $\lambda<0$.
Denote by $\Gamma_+\subset\R_+$ the additive monoid generated by all $\lambda\ge0$ such that $X_\lambda\ne0$, and  recall that $\Gamma$ is the group generated by $\{\lambda\in\R: A_\lambda\ne0 \}$. As we will see in examples later, $\Gamma_+$ can be much smaller than $\Gamma\cap\R_+$.

\begin{definition}
The right C$^*$-Hilbert $A_c$-module
\[
\F=\bigoplus_{\lambda\in\Gamma_+}X_\lambda
\]
is called the \emph{Fock module}. %, or the \emph{Fock C$^*$-correspondence}, when we want to stress that we have a left action of $A$ on $\F$.
Equivalently, we can complete $A$ to a right C$^*$-Hilbert $A_0$-module $X$ using the inner product $\langle a,b\rangle=E(a^*b)$, then $\F=X\otimes_{A_0}(A_0/I)$.
\end{definition}

Denote by $\LL(\F)$ the C$^*$-algebra of adjointable operators on $\F$. The left action of $A$ on itself defines a $*$-homomorphism $\Lambda\colon A\to\LL(\F)$. We will usually write $a\xi$ instead of $\Lambda(a)\xi$ for $a\in A$ and $\xi\in\F$.

\begin{example}\label{ex:TP2}
Let us see the meaning of this definition in the setting of Example~\ref{ex:TP}, that is, for a Pimsner--Toeplitz algebra $A=\TT_Y$ with the dynamics $\sigma=\sigma^H$ defined by an operator $H$ with strictly positive and pure point spectrum. As we showed in that example, the crystal of $(\TT_Y,\sigma^H)$ can be canonically identified with the coefficient algebra $B$.

If $\xi\in Y_\lambda$ for some $\lambda>0$, then $S_\xi^*\in A_{-\lambda}$ has image zero in $\F$. It follows that all elements of the form $S_{\xi_1}\dots S_{\xi_n}S^*_{\zeta_1}\dots S^*_{\zeta_m}$, with $m\ge1$, have zero image in $\F$. From this we can conclude that $\Gamma_+$ is the monoid generated by the pure point spectrum of $H$ and, for every $\lambda\in\Gamma_+\setminus\{0\}$, we have a unitary isomorphism
\[
\bigoplus^\infty_{n=1} \smashoperator[r]{\bigoplus_{\substack{\lambda_1,\dots,\lambda_n>0 \\ \lambda_1+\dots+\lambda_n=\lambda}}} \ Y_{\lambda_1}\otimes_B\dots\otimes_B Y_{\lambda_n}\cong X_\lambda,\quad \xi_1\otimes\dots\otimes\xi_n\mapsto S_{\xi_1}\dots S_{\xi_n}+\overline{A_\lambda I},
\]
of right C$^*$-Hilbert $B$-modules. If follows that the right C$^*$-Hilbert $B$-module $\F$ is isomorphic to the usual Fock module $\F_Y$. This isomorphism respects the left actions of $A=\TT_Y$.
\ee
\end{example}

Next we want to characterize representations induced by $\F$, cf.~\cite{MR4150892}*{Section~6}. By a representation we will always mean a nondegenerate representation.

\begin{definition}
A representation $\pi\colon A\to B(H)$ is called a \emph{vacuum representation} if $\pi(A)H_0$ is dense in $H$, where
\[
H_0=\{\Omega\in H\colon \text{$\pi(a)^*\Omega=0$ for all $a\in A_\lambda$ and $\lambda>0$}\}
\]
is the subspace of \emph{vacuum vectors}.
\end{definition}

\begin{proposition} \label{prop:vacuum-rep}
A representation $\pi\colon A\to B(H)$ is a vacuum representation if and only if it is induced from a representation $\rho$ of $A_c$ by the Fock module. Furthermore, the map $\rho\mapsto\Ind\rho$ defines an equivalence between the category of representations of $A_c$ and the category of vacuum representations of~$A$.
\end{proposition}

\bp
Assume $\pi\colon A\to B(H)$ is a vacuum representation. Consider the subspace $H_0\subset H$ of vacuum vectors. This subspace is invariant under $A_0$ and the action of $I\subset A_0$ on this subspace is zero, so we get a representation $\rho\colon A_c\to B(H_0)$. Observe next that if $a\in A_\lambda$ and $b\in A_\mu$ for some $\lambda\ne\mu$, then $\pi(a)H_0\perp\pi(b)H_0$, since if, for example, $\lambda<\mu$, then  $a^*b\in A_{\mu-\lambda}$ and hence $\pi(b)^*\pi(a)H_0=\pi(a^*b)^*H_0=0$. It follows that we have a well-defined isometric map
\[
\F\otimes_{A_c}H_0\to H,\quad (a+\overline{A_\lambda I})\otimes\Omega\mapsto\pi(a)\Omega\quad\text{for}\quad a\in A_\lambda,\ \lambda\in\R.
\]
This map obviously intertwines the left actions of $A$. Since by assumption $\pi$ is a vacuum representation, this map is unitary, hence $\pi\sim\Ind\rho=\F\text{-}\Ind^A_{A_c}\rho$.

Note that the construction of $\rho$ gives rise to a functor $F$ from the category of vacuum representations of $A$ to the category of representations of $A_c$, and what we have shown is that $\Ind\circ F$ is naturally isomorphic to the identity functor.

Conversely, it is easy to see that if $\pi=\Ind\rho$ for a representation $\rho\colon A_c\to B(H'_0)$ of $A_c$, then $\pi$ is a vacuum representation.
To finish the proof of the proposition we have to show existence of a natural isomorphism $\rho\cong F(\Ind\rho)$. Let $H=\F\otimes_{A_c}H'_0$ be the Hilbert space underlying the representation $\pi=\Ind\rho$. As $X_0=A_c$, we can identify $H'_0$ with a subspace of $H$. Then it suffices to show that $H'_0$ is exactly the subspace $H_0\subset H$ of vacuum vectors. Clearly, $H'_0\subset H_0$. For the opposite inclusion, take $\Omega'\notin H'_0$. Then $\Omega'$ cannot be orthogonal to all the subspaces $X_\lambda\otimes_{A_c}H'_0=\overline{\pi(A_\lambda)H'_0}$, $\lambda>0$. Hence there exists $\Omega\in H'_0$ and $a\in A_\lambda$, with $\lambda>0$, such that $(\Omega',\pi(a)\Omega)\ne0$. But then $\pi(a)^*\Omega'\ne0$, so~$\Omega'$ is not a vacuum vector.
\ep

\begin{corollary}\label{cor:Ac-rep-theory}
We have:
\begin{enumerate}
\item $A_c=0$ if and only if $A$ has no nonzero vacuum representations;
\item $A_c=\C$ if and only if $A$ has a unique up to equivalence nonzero representation with a cyclic vacuum vector;
\item assuming that $A$ is separable, we have $A_c\cong\K(H)$ for a nonzero Hilbert space $H$ if and only if $A$ has a unique up to equivalence irreducible vacuum representation.
\end{enumerate}
\end{corollary}

\bp
Part (1) is obvious. For part (2), observe that a vacuum vector $\Omega$ is cyclic for $A$ if and only if~$\Omega$ is cyclic for the representation of $A_c$ on the space of vacuum vectors. Then (2) follows by observing that $\C$ is the only C$^*$-algebra that has exactly one cyclic representation up to equivalence. Similarly, (3) holds because the algebras $\K(H)$ are the only separable C$^*$-algebras that have exactly one irreducible representation up to equivalence~\cite{MR0057477}.
\ep

\begin{example}\label{ex:CAR}
Consider the CAR-algebra $A=\CAR(H)$ over a Hilbert space $H$. By definition~\cite{BR2}, it is generated by elements $a(\xi)$, $\xi\in H$, such that the map $\xi\mapsto a(\xi)$ is anti-linear and
\begin{equation}\label{eq:CAR}
a(\xi)a(\zeta)+a(\zeta)a(\xi)=0,\qquad a(\xi)a(\zeta)^*+a(\zeta)^*a(\xi)=(\zeta,\xi)1.
\end{equation}
Consider the gauge action $\gamma_t(a(\xi))=e^{-it}a(\xi)$.

Given a representation $\pi\colon\CAR(H)\to B(K)$, a unit vector $\Omega\in K$ is a vacuum vector if and only if $\pi(a(\xi))\Omega=0$ for all $\xi\in H$. If $\Omega$ is in addition cyclic, then the second relation in~\eqref{eq:CAR} implies that $K$ is spanned by vectors of the form $\pi(a(\xi_1)^*\dots a(\xi_n)^*)\Omega$. The same relation implies that the scalar products of such vectors are uniquely determined. Therefore there is at most one representation with a cyclic unit vacuum vector. The standard representation of $\CAR(H)$ by creation and annihilation operators on the fermionic Fock space does give an example of such a representation. By Corollary~\ref{cor:Ac-rep-theory} we therefore conclude that the crystal of $(\CAR(H),\gamma)$ is $\C$.

\smallskip

Note that the above argument used only the second relation in~\eqref{eq:CAR}. More generally, we can consider a Wick algebra $\WW(T)$ in the sense of~\cite{MR1359923}, which has generators $a_j$, $j\in J$, and relations of the form
\[
a_ia_j^*=\delta_{ij}1+\sum_{k,l}T^{kl}_{ij}a_l^*a_k.
\]
Here $T^{kl}_{ij}$ are complex numbers and in order to simplify matters we assume that for fixed $i$ and $j$ only finitely many of them are nonzero. We define the gauge action by $\gamma_t(a_j)=e^{-it}a_j$. Then, similarly to the CAR-algebra case, we can conclude that either the crystal of $(\WW(T),\gamma)$ is~$\C$ and there is a unique up to equivalence nonzero vacuum representation of $\WW(T)$ with a cyclic vacuum vector, or the crystal of $(\WW(T),\gamma)$ is zero and $\WW(T)$ has no nonzero vacuum representations.
\ee
\end{example}

The proof of Proposition~\ref{prop:vacuum-rep} also implies the following.

\begin{corollary} \label{cor:central-proj}
Given a representation $\pi\colon A\to B(H)$, consider the von Neumann algebra $M=\pi(A)''$. Then there exists a central projection $z\in M$ such that $zH$ is the space of the largest subrepresentation of~$\pi$ induced from $A_c$ by the Fock module $\F$.
\end{corollary}

\bp
The proof of Proposition~\ref{prop:vacuum-rep} shows that if $H_0\subset H$ is the subset of vacuum vectors, then the representation of $A$ on $\overline{\pi(A)H_0}$ is induced by the Fock module. Clearly, this is the largest subrepresentation of $\pi$ induced by $\F$. Since $H_0$ is invariant under $\pi(A)'$, the subspace $\overline{\pi(A)H_0}$ is invariant under both $M=\pi(A)''$ and $M'=\pi(A)'$, hence the projection $z$ onto this subspace lies in the center of $M$.
\ep

Let us now look at the topological properties of the Fock module.

\begin{proposition}\label{prop:density}
The image $\Lambda(A)$ of $A$ in $\LL(\F)=M(\K(\F))$ is strictly dense.
\end{proposition}

\bp
Denote by $Q_0$ the orthogonal projection $\F\to X_0=A_c$. Let us show first that $Q_0$ lies in the strict closure of $\Lambda(A_0)$. For this it suffices to show that given $\eps>0$ and vectors $\xi_i\in X_{\lambda_i}$ with $\lambda_i\ge0$, $1\le i\le n$, we can find $a\in A_0$ such that $0\le a\le1$, $\|a \xi_i\|<\eps$ if $\lambda_i>0$, and $\|\xi_i-a\xi_i\|<\eps$ if $\lambda_i=0$.

If $\lambda_i=0$ for all $i$, then as $a$ we can take an appropriate element of an approximate unit in~$A_0$, while if $\lambda_i>0$ for all $i$, then we can take $a=0$. Let us therefore assume that $\lambda_1,\dots,\lambda_m>0$ and $\lambda_{m+1},\dots,\lambda_n=0$ for some $1\le m< n$. Choose $e\in A_0$ such that $0\le e\le1$ and $\|\xi_i-e\xi_i\|<\eps$ for $i=m+1,\dots,n$. Next, for every $\lambda>0$, the left action of $I_{\lambda}$ on $X_{\lambda}$ defines a surjective $*$-homomorphism $I_{\lambda}\to\K(X_{\lambda})$. By choosing approximate units in the C$^*$-algebras $\K(X_{\lambda_i})$ we can therefore find one by one elements $f_i\in I_{\lambda_i}$, $1\le i\le m$, such that $0\le f_i\le 1$ and
\[
\|(1-f_i)(1-f_{i-1})\dots(1-f_1)e^{1/2}\xi_i\|<\eps\quad\text{for}\quad i=1,\dots,m.
\]
Since the elements $f_i$ act by zero on $X_0$, we can then take
\[
a=e^{1/2}(1-f_1)\dots(1-f_{m-1})(1-f_m)(1-f_{m-1})\dots(1-f_1)e^{1/2}.
\]
Hence $Q_0$ indeed lies in the strict closure of $\Lambda(A_0)$.

Now, observe that if $a\in A_\lambda$ and $b\in A_\mu$ for some $\lambda,\mu\ge0$, then $\Lambda(a)Q_0\Lambda(b)^*=\theta_{\bar a,\bar b}$, where $\bar a,\bar b$ are the images of $a$ and $b$ in $X_\lambda$ and $X_\mu$, resp., and $\theta_{\bar a,\bar b}\in\K(\F)$ is defined by $\theta_{\bar a,\bar b}\xi=\bar a\langle\bar b,\xi\rangle$. This implies that the strict closure of $\Lambda(A)$ contains~$\K(\F)$, hence it coincides with~$\LL(\F)$.
\ep

\begin{corollary}\label{cor:vacuum-rep}
A representation $\pi\colon A\to B(H)$ is a vacuum representation if and only if it factors through $\Lambda(A)\subset\LL(\F)$ and the representation of $\Lambda(A)$ we thus get is strict-strong continuous, that is, it is continuous with respect to the strict topology on $\LL(\F)=M(\K(\F))$ and the strong operator topology on $B(H)$.
\end{corollary}

\bp
If $\pi\colon A\to B(H)$ is induced from a representation of $A_c$ by the Fock module, then by construction there exists a representation $\tilde\pi\colon\LL(\F)\to B(H)$ such that $\pi=\tilde\pi\circ\Lambda$ and $\tilde\pi|_{\K(\F)}$ is nondegenerate. Then $H=\tilde\pi(\K(\F))H$ by the Cohen--Hewitt factorization theorem, and hence $\tilde\pi|_{\Lambda(A)}$ is strict-strong continuous.

Conversely, assume $\pi=\tilde\pi\circ\Lambda$ for a strict-strong continuous representation $\tilde\pi\colon \Lambda(A)\to B(H)$. By Proposition~\ref{prop:density} and a version of Kaplansky's density theorem for multiplier algebras, the unit ball of $\Lambda(A)$ is strictly dense in the unit ball of $\LL(F)$. Hence $\tilde\pi$ extends by continuity to a representation of $\LL(\F)$ that is strict-strong continuous on the norm-bounded sets. We continue to denote this representation by $\tilde\pi$. Since the right C$^*$-Hilbert $A_c$-module $\F$ is full, the representation $\tilde\pi|_{\K(\F)}$ is induced from a representation of $A_c$, and hence the same is true for $\pi$.
\ep

Finally, let us look at consequences of the above results for states on $A$. The following terminology is motivated by~\cite{MR1953065}, see also~\cite{MR4150892}.

\begin{definition}
A positive linear functional $\phi$ on $A$ is said to be of \emph{finite type}, if the corresponding GNS-representation is a vacuum representation. It is said to be of \emph{infinite type}, if the corresponding GNS-representation has no nonzero vacuum subrepresentations.
\end{definition}

Note that if $\phi$ is of finite type, then the cyclic vector in the corresponding GNS-representation is not necessarily a vacuum vector.

Using Corollary~\ref{cor:vacuum-rep}, the same arguments as in the proof of~\cite{MR4150892}*{Lemma~6.5} show that:
\begin{itemize}
\item $\phi$ is of finite type if and only if $\phi=\psi\circ\Lambda$ for a strictly continuous positive linear functional~$\psi$ on $\Lambda(A)\subset\LL(\F)$;
\item $\phi$ is of infinite type if and only if there is no nonzero strictly continuous positive linear functional $\psi$ on $\Lambda(A)$ such that $\phi\ge\psi\circ\Lambda$.
\end{itemize}
Arguing as in the proofs of~\cite{MR4150892}*{Proposition~6.6 and Corollary~6.7} we then get the following Wold-type decomposition of positive linear functionals.

\begin{proposition} \label{prop:Wold-state}
Every positive linear functional $\phi$ on $A$ has a unique decomposition $\phi=\phi_f+\phi_\infty$ where $\phi_f$ is of finite type and $\phi_\infty$ is of infinite type. If $\phi$ satisfies the $\sigma$-KMS$_\beta$-condition for some $\beta\in\R$, then the same is true for $\phi_f$ and $\phi_\infty$.
\end{proposition}

Explicitly, if $(H_\phi,\pi_\phi,\xi_\phi)$ is the GNS-triple associated with $\phi$ and $z\in\pi_\phi(A)''$ is the central projection given by Corollary~\ref{cor:central-proj}, then
\[
\phi_f=(\pi_\phi(\stub)z\xi_\phi,\xi_\phi),\qquad \phi_\infty=(\pi_\phi(\stub)(1-z)\xi_\phi,\xi_\phi).
\]

\subsection{KMS states of finite type}
Our next goal is to show that KMS-states of finite type are completely determined by the crystal. To formulate the result we first recall the construction of induced traces~\cite{MR2056837}.

Assume $B$ is a C$^*$-algebra, $Y$ is a right C$^*$-Hilbert $B$-module and  $\tau$ is a tracial positive linear functional on $B$. Then there is a unique strictly lower semicontinuous, in general infinite, trace $\Tr^Y_\tau$ on $\LL(Y)$ %, which we often denote simply by~$\Tr_\tau$,
such that
\[
\Tr^Y_\tau(\theta_{\xi,\xi})=\tau(\langle\xi,\xi\rangle) \quad \text{for all} \quad \xi\in Y,
\]
see \cite{MR2056837}*{Theorem 1.1}. Explicitly, $\Tr^Y_\tau$ can be defined as follows. Assume $(u_i)_{i\in I}$ is an approximate unit in $\K(Y)$ such that, for every $i$, $u_i=\sum_{\xi\in J_i}\theta_{\xi,\xi}$ for some finite set $J_i\subset Y$. Then
\begin{equation}\label{eq:Tr}
\Tr^Y_\tau(T)=\sup_i\sum_{\xi\in J_i}\tau(\langle\xi,T\xi\rangle)=\lim_i\sum_{\xi\in J_i}\tau(\langle\xi,T\xi\rangle)\quad \text{for}\quad T\in\LL(Y)_+.
\end{equation}

We will write $\dim_\tau Y$ for $\Tr^Y_\tau(1)$. Let us record a few properties of this dimension function.

\begin{lemma}\label{lem:dimension}
Assume $\tau$ is a tracial state on $B$. Then we have:
\begin{enumerate}
  \item if $(H_\tau,\pi_\tau,\xi_\tau)$ is the GNS-triple associated with $\tau$ and $M=\pi_\tau(B)''$, then $\dim_\tau Y$ is the von Neumann dimension of the right Hilbert $M$-module $Y\otimes_B H_\tau$ with respect to the tracial state $(\stub \xi_\tau,\xi_\tau)$ on $M$;
  \item if $Y$ is topologically generated by $n$ elements, then $\dim_\tau Y\le n$.
\end{enumerate}
\end{lemma}

\bp
It suffices to prove (1) for C$^*$-Hilbert modules $Y$ topologically generated by finitely many elements $\xi_1,\dots,\xi_n$. For this consider the bounded right $M$-module map
\[
T\colon L^2(M,\tau)^n=H_\tau^n\to Y\otimes_B H_\tau,\quad (\Lambda_\tau(a_i))_i\mapsto \sum_i\xi_i\otimes \Lambda_\tau(a_i),
\]
where $\Lambda_\tau\colon B\to H_\tau$ is the GNS-map, that is, $\Lambda_\tau(a)=\pi_\tau(a)\xi_\tau$. Then
\[
TT^*=\sum_i\theta_{\xi_i,\xi_i}\otimes1\quad\text{and}\quad T^*T=(\pi_\tau(\langle\xi_i,\xi_j\rangle))_{i,j}.
\]

Consider the polar decomposition $T^*=u|T^*|$. Then $u$ is an isomorphism of $Y\otimes_B H_\tau$ onto a right $M$-submodule of $L^2(M,\tau)^n$ and by definition the von Neumann dimension of $Y\otimes_B H_\tau$ is the trace of the projection $uu^*\in\Mat_n(M)$. We need to show that this trace equals $\dim_\tau Y$. For this consider the operators
\[
a_k=\frac{k^{1/2}1}{\left(1+k\sum_i\theta_{\xi_i,\xi_i}\right)^{1/2}}\in\LL(Y)\quad \text{and}\quad T_k=(a_k\otimes 1)T=\frac{k^{1/2}}{\left(1+k|T^*|^2\right)^{1/2}}T
\]
for $k\ge1$. Then, on the one hand, $T_k^*T_k$ converges strongly to $uu^*$ and hence the von Neumann dimension of $Y\otimes_B H_\tau$ equals
\[
\lim_{k\to\infty}(\Tr\otimes\tau)(T_k^*T_k)=\lim_{k\to \infty}\sum_i\tau(\langle a_k\xi_i,a_k\xi_i\rangle).
\]
On the other hand, since $\sum_i\theta_{\xi_i,\xi_i}$ is strictly positive in $\K(Y)$, the elements
\[
\sum_i a_k\theta_{\xi_i,\xi_i}a_k=\sum_i\theta_{a_k\xi_i,a_k\xi_i}
\]
form an approximate unit in $\K(Y)$ and hence the last limit above equals $\Tr_\tau^Y(1)=\dim_\tau Y$. This proves part (1). Part (2) is an immediate consequence of this proof.
\ep

Returning to an almost periodic dynamical system $(A,\sigma)$, the automorphisms $\sigma_t$ of $A$ extend to unitaries $U_t$ on the Fock module $\F$. Let $D$ be the generator of $(U_t)_{t\in\R}$, so $U_t=e^{itD}$, where
\[
Dx=\lambda x\quad\text{for}\quad x\in X_\lambda,\ \lambda\in\Gamma_+.
\]
For every $\beta\in\R$ and a positive trace $\tau$ on $A_c$ we then get an $(\Ad U)$-KMS$_\beta$-weight
\[
\Tr^\F_\tau(e^{-\beta D/2}\stub e^{-\beta D/2})
\]
on~$\K(\F)$, which we will write as $\Tr^\F_\tau(\stub e^{-\beta D})$. It extends to a strictly lower semicontinuous weight on $\LL(\F)$, which we continue to denote by $\Tr^\F_\tau(\stub e^{-\beta D})$. We remark that if $\Tr^\F_\tau$ is infinite or $e^{-\beta D}$ is unbounded, then $\Tr^\F_\tau(\stub e^{-\beta D})$ should rather be interpreted as the KMS$_\beta$-weight induced by $\tau$ using the equivariant C$^*$-correspondence $\F$, see \cite{MR2056837}*{Section~3}.

If $\Tr^\F_\tau(\stub e^{-\beta D})$ happens to be a state, that is,
\[
\Tr^\F_\tau(e^{-\beta D})=\sum_{\lambda\in\Gamma_+}e^{-\beta\lambda}\dim_\tau{X_\lambda}=1,
\]
then $\phi=\Tr^\F_\tau(\Lambda(\stub)e^{-\beta D})$ is a $\sigma$-KMS$_\beta$-state on $A$, as $(\Ad U_t)\circ\Lambda=\Lambda\circ\sigma_t$. Note that if $a\in A_\mu$, then
\[
\phi(a)=\begin{dcases*}
\sum_{\lambda\in\Gamma_+}e^{-\beta\lambda}\Tr^{X_\lambda}_\tau(\Lambda(a)|_{X_\lambda}),& if $\mu=0$,\\
0,& otherwise.
\end{dcases*}
\]

By construction the state $\phi$ is a KMS-state of finite type. It turns our this way we get all such states.

\begin{theorem}\label{thm:KMS-finite-type}
For every $\beta\in\R$, the map
\[
\tau\mapsto\Tr^\F_\tau(\Lambda(\stub)e^{-\beta D})
\]
is an affine bijection between the positive traces $\tau$ on $A_c$ such that $\sum_{\lambda\in\Gamma_+}e^{-\beta\lambda}\dim_\tau X_\lambda=1$
and the $\sigma$-KMS$_\beta$-states on $A$ of finite type.
\end{theorem}

\bp
Assume $\phi$ a $\sigma$-KMS$_\beta$-state on $A$ of finite type. Let $\psi$ be the unique strictly continuous state on $\LL(\F)$ such that $\phi=\psi\circ\Lambda$. Since $\Lambda(A)$ is strictly dense in $\LL(\F)$, the state $\psi$ satisfies the $(\Ad U)$-KMS$_\beta$-condition. Hence, by \cite{MR2056837}*{Theorem~3.2}, the  $(\Ad U)$-KMS$_\beta$-state $\psi|_{\K(\F)}$ is induced by a uniquely defined semifinite lower semicontinuos trace $\tau$ on $A_c$, so that $\psi=\Tr^\F_\tau(\stub e^{-\beta D})$ on $\LL(\F)$. As $\Tr^\F_\tau(e^{-\beta D})=1$ and $X_0=A_c$, the trace $\tau$ must be finite. This proves the stated bijection. It is also clear that the map $\tau\mapsto\Tr^\F_\tau(\Lambda(\stub)e^{-\beta D})$ is affine.
\ep

Note that if a  tracial state $\tau$ on $A_c$ satisfies $\sum_{\lambda\in\Gamma_+}e^{-\beta\lambda}\dim_\tau X_\lambda<\infty$ for some $\beta$, then it satisfies the same condition for all larger $\beta$'s. Therefore we can say that every phase defined by an equilibrium state of finite type persists all the way to zero temperature. It follows that the ground state corresponding to $\tau$ under the bijection given by Proposition~\ref{prop:ground} is a so-called KMS$_\infty$-state. The study of such states is a separate interesting topic that we do not touch upon in this paper, see~\cite{MR4673062} and the references there.

\subsection{KMS states for low temperatures}

Our next goal is to show that under additional assumptions all KMS-states are of finite type below some temperature.

\begin{theorem}\label{thm:KMS-full}
Assume that there exist closed right $A_0$-submodules $B_\lambda\subset A_\lambda$, $\lambda>0$, and a number $\beta_0\ge0$ such that for every $\beta>\beta_0$ we have:
\begin{enumerate}
\item $\sum_{\lambda>0}B_\lambda A_{\mu-\lambda}$ is dense in $A_\mu$ for every $\mu>0$;
\item $\sum_{\lambda>0}e^{-\beta\lambda}\dim_\tau B_\lambda<1$ for every tracial state $\tau$ on $A_0$.
\end{enumerate}
Then, for every $\beta>\beta_0$, all $\sigma$-KMS$_\beta$-states on $A$ are of finite type, so by Theorem~\ref{thm:KMS-finite-type} they are in a one-to-one correspondence with the positive traces $\tau$ on $A_c$ such that $\sum_{\lambda\in\Gamma_+}e^{-\beta\lambda}\dim_\tau X_\lambda=1$.
\end{theorem}

\bp
Fix $\beta>\beta_0$. By Proposition~\ref{prop:Wold-state} it suffices to show that there are no $\sigma$-KMS$_\beta$-states of infinite type. Equivalently, we have to show that if $\phi$ is a  $\sigma$-KMS$_\beta$-state and $(H_\phi,\pi_\phi,\xi_\phi)$ is the corresponding GNS-triple, then the subspace $H_0\subset H_\phi$ of vacuum vectors is nonzero.

Put $M=\pi_\phi(A)''$. We continue to denote the faithful state $(\stub\xi_\phi,\xi_\phi)$ on $M$ by $\phi$. Consider its centralizer $M_\phi$ and the tracial state $\tau=\phi|_{M_\phi}$. We have $\pi_\phi(A_0)\subset M_\phi$, and will use the same symbol~$\tau$ for the trace $\tau\circ\pi_\phi$ on $A_0$.

Denote by $Q_0$ the projection onto $H_0$. For every $\lambda>0$, denote by $Q_\lambda$ the projection onto $\overline{\pi_\phi(B_\lambda) H_\phi}$. We claim that
\begin{equation}\label{eq:Q0}
Q_0=1-\bigvee_{\lambda>0}Q_\lambda.
\end{equation}
In other words, we claim that a vector $\Omega\in H_\phi$ is a vacuum vector if and only if $\Omega\perp \pi(B_\lambda)H_\phi$ for all~$\lambda>0$. The ``only if'' part is clear.

For the ``if'' part, assume $\Omega$ is a vector satisfying $\Omega\perp \pi(B_\lambda)H_\phi$ for all~$\lambda>0$. Take $a\in A_\mu$ for some $\mu>0$. By assumption~(1), we can find a sequence $(a_n)_n$ in $\sum_{\lambda>0}B_\lambda A_{\mu-\lambda}$ converging to~$a$, which implies that $\pi_\phi(a_n)\xi_\phi\to\pi_\phi(a)\xi_\phi$. Since the vector $\xi_\phi$ is separating for $M$ by the KMS-condition, hence cyclic for $M'$, it follows that $\pi_\phi(a_n)\xi\to\pi_\phi(a)\xi$ for a dense set of vectors $\xi\in H_\phi$. As $\pi_\phi(a_n)\xi\in\sum_{\lambda>0}\pi_\phi(B_\lambda)H_\phi$ and hence $\Omega\perp \pi_\phi(a_n)\xi$, we conclude that $\Omega\perp\pi_\phi(a)\xi$ for a dense set of vectors $\xi$. But this is equivalent to $\pi_\phi(a)^*\Omega=0$. Hence $\Omega$ is indeed a vacuum vector. This proves~\eqref{eq:Q0}.

Next, we claim that for every $\lambda>0$ we have $Q_\lambda\in M_\phi$ and
\begin{equation}\label{eq:trace-Q-lambda}
\tau(Q_\lambda)=e^{-\beta\lambda}\dim_\tau B_\lambda.
\end{equation}
Indeed, we can find an approximate unit $(e_i)_i$ in $\overline{B_\lambda B_\lambda^*}\cong \K(B_\lambda)$ of the form
\[
e_i=\sum_{a\in J_i}aa^*=\sum_{a\in J_i}\theta_{a,a}
\]
for finite subsets $J_i\subset B_\lambda$. Then $\pi_\phi(e_i)\to Q_\lambda$ strongly, which implies that $Q_\lambda\in M_\phi$. By the KMS-condition we also have
\[
\tau(e_i)=\sum_{a\in J_i}\tau(aa^*)=e^{-\beta\lambda}\sum_{a\in J_i}\tau(a^*a)=e^{-\beta\lambda}\sum_{a\in J_i}\Tr^{B_\lambda}_\tau(\theta_{a,a})=e^{-\beta\lambda}\Tr^{B_\lambda}_\tau(e_i),
\]
cf.~\cite{MR2056837}, which in the limit gives~\eqref{eq:trace-Q-lambda}.

Now, since $\tau$ is a normal trace on $M_\phi$, we have $\tau(p\vee q)\le\tau(p)+\tau(q)$ for any projections $p,q\in M_\phi$. Hence, combining \eqref{eq:Q0} and \eqref{eq:trace-Q-lambda} and using assumption (2), we get
\[
\tau(Q_0)\ge 1-\sum_{\lambda>0}e^{-\beta\lambda}\dim_\tau B_\lambda>0,
\]
proving that $Q_0\ne0$.
\ep

\begin{remark}
The assumptions of the theorem can be slightly relaxed as follows, although for now we do not see much benefit in doing this. First, condition (2) is needed only for tracial states on $A_0$ that are restrictions of KMS$_\beta$-states. Such traces have additional properties that might be useful in estimating dimensions. For example, we must have $\tau(aa^*)=e^{-\beta\lambda}\tau(a^*a)$ for all $a\in A_\lambda$. Second, from the proof of the theorem we see that instead of (1) it suffices to require $\sum_\lambda B_\lambda A_{\mu-\lambda}$ to be dense in $A_\mu$ with respect to every semi-norm $\lVert\stub\rVert_\tau$ on $A_\mu$ defined by
\begin{equation}\label{eq:semi-tau}
\|a\|_\tau=\tau(a^*a)^{1/2}.
\end{equation}
Note that if $(H_\tau,\pi_\tau,\xi_\tau)$ is the GNS-triple associated with $\tau$, then this is the semi-norm obtained on $A_\mu$ using the map $A_\mu\to A_\mu\otimes_{A_0}H_\tau$, $a\mapsto a\otimes\xi_\tau$.\ee
\end{remark}

It is an interesting problem to characterize the smallest number $\beta_c$ above which all KMS$_\beta$-states are of finite type, cf.~\cite{MR4150892}*{Theorem~7.1}.
The number $\beta_0$ given by Theorem~\ref{thm:KMS-full} depends on the choice of $B_\lambda$ and provides only a very rough upper bound for $\beta_c$.

\smallskip

When the trace state space of $A_0$ is complicated, the inequality in Lemma~\ref{lem:dimension}(2) might be the only viable way to check the conditions of Theorem~\ref{thm:KMS-full}. A possible strategy to verify these conditions is therefore to look for a smallest set $\bigcup_{\lambda>0} S_\lambda$, $S_\lambda\subset A_\lambda$, such that $\sum_{\lambda>0}S_\lambda A_{\mu-\lambda}$ is dense in $A_\mu$ for all $\mu>0$, and then let $B_\lambda=\overline{S_\lambda A_0}$.

\begin{example}
Consider again a Pimsner--Toeplitz algebra $A=\TT_Y$ with a dynamics $\sigma=\sigma^H$ as in Examples~\ref{ex:TP}, \ref{ex:TP2}. For $\lambda>0$, let $B_\lambda$ be the closed right $A_0$-module generated by the elements $S_\xi$ with $\xi\in Y_\lambda$. Since the elements $S_{\xi_1}\dots S_{\xi_n}S^*_{\zeta_1}\dots S^*_{\zeta_m}$ span a dense subspace of $\TT_Y$, condition (1) in Theorem~\ref{thm:KMS-full} is obviously satisfied. As for condition (2), note that we have an isomorphism
\[
Y_\lambda\otimes_B A_0\cong B_\lambda, \quad \xi\otimes a\mapsto S_\xi a.
\]
This implies that for every tracial state $\tau$ on $A_0$ we have
\[
\dim_\tau B_\lambda=\dim_{\tau|_B}Y_\lambda.
\]
It follows that condition (2) is satisfied when
\[
\Tr^Y_\tau(e^{-\beta H})=\sum_{\lambda>0}e^{-\beta\lambda}\dim_\tau Y_\lambda<1
\]
for all tracial states $\tau$ on $B$. Theorem~\ref{thm:KMS-full} asserts that then all KMS$_\beta$-states are of finite type, that is, by Example~\ref{ex:TP2}, they have the form $\Tr_\tau^{\F_Y}(\stub e^{-\beta D})$ for a positive trace $\tau$ on $B$, where $D=0\oplus H\oplus (H\otimes1+1\otimes H)\oplus\dots$.

This is consistent with results in \cite{MR2056837}*{Section~2}, where it is shown (under the assumption $\overline{BY}=Y$) that a $\sigma^H$-KMS$_\beta$-state of infinite type on $\TT_Y$ exists if and only if $\Tr^Y_\tau(\stub e^{-\beta H})=\tau$ on $B$ for some tracial state $\tau$.

Note that this example shows that condition (2) in Theorem~\ref{thm:KMS-full} cannot be relaxed to
\[
\sum_{\lambda>0}e^{-\beta\lambda}\dim_\tau B_\lambda<\infty,
\]
as might be suggested by~\cite{MR4150892}*{Theorem~7.1} or~\cite{NS}*{Theorem~4.5}.
\end{example}

\begin{example}\label{ex:LCM}
We next look at KMS-states on C$^*$-algebras of right LCM semigroups. This class of algebras have over years received significant attention, see, e.g., \cites{MR2671177,MR3932591} and references there. A general formula for KMS$_\beta$-states on these algebras for large $\beta$ has been obtained in~\cite{NS} using the groupoid approach. As we will see, it also arises naturally from the results we have proved.

So let $S$ be a right LCM semigroup, by which we mean a left cancellative monoid with identity element~$e$ such that for all $s,t\in S$ we have either $sS\cap tS=\emptyset$ or $sS\cap tS=rS$ for some $r\in S$. %If $t\in sS$ for some $s,t\in S$, we denote by $s^{-1}t$ the unique element of $S$ such that $s(s^{-1}t)=t$.
The semigroup C$^*$-algebra $C^*(S)$ is generated by isometries $v_s$, $s\in S$, satisfying the relations
\[
v_e=1,\quad v_sv_t=v_{st},\quad \text{and}\quad v_sv_s^*v_tv_t^*=
\begin{cases}
v_rv_r^*,&\text{if $sS\cap tS=rS$},\\
0,&\text{if $sS\cap tS=\emptyset$}.
\end{cases}
\]
The elements $v_sv_t^*$, for all $s,t\in S$, span a dense subspace of $C^*(S)$.

Consider the dynamics $\sigma^N$ on $C^*(S)$ defined by a semigroup homomorphism $N\colon S\to([1,+\infty),\cdot)$:
\[
\sigma^N_t(v_s)=N(s)^{it}v_s.
\]
We assume that $N$ satisfies the following condition:
\begin{equation}\label{eq:scale-condition}
\text{for any}\quad s\in S,\ t\in\ker N,\quad \text{we have}\quad sS\cap tS=rS\quad \text{for some} \quad r\in s(\ker N),
\end{equation}
which by~\cite{NS}*{Lemma~3.1} is necessary for existence of a KMS$_\beta$-state for some $\beta>0$. This implies that $\ker N$ is a hereditary LCM submonoid of $S$.

Since the elements $v_sv_t^*$ span a dense subalgebra of $A=C^*(S)$, the spectral space $A_\lambda$ is the closure of the linear span of $v_sv_t^*$ with $\log N(s)-\log N(t)=\lambda$. Then the ideal $I$ is the closure of the linear span of $v_sv_t^*$ with $N(s)=N(t)>1$.

By~\cite{NS}*{Corollary~4.4}, we can view $C^*(\ker N)$ as a subalgebra of $C^*(S)$ and there is a conditional expectation $E\colon C^*(S)\to C^*(\ker N)$ that kills $v_sv_t^*$ if either $s$ or $t$ is not in $\ker N$. This implies that
\[
A_0=C^*(\ker N)\oplus I
\]
as a vector space, hence the crystal of $(C^*(S),\sigma^N)$ is $C^*(\ker N)$.

The group $\Gamma\subset\R$ is generated by $\log N(S)$. Since $v_tv_t^*\in I$ for $N(t)>1$, the image of $v_sv_t^*$ in the Fock module $\F$ is zero for every such $t$. Therefore $\Gamma_+=\log N(S)$, which is in general different from $\Gamma\cap\R_+$. Furthermore, the right $A_c$-modules $X_\lambda$ are topologically generated by the images of elements $v_s$ with $\log N(s)=\lambda$.

To apply Theorem~\ref{thm:KMS-full}, for $\lambda\notin\log N(S)$ we take $B_\lambda=0$. For $\lambda\in\log N(S)$, $\lambda>0$, we take $B_\lambda$ to be the closed linear span of $v_sA_0$ for all $s$ with $\log N(s)=\lambda$. Condition (1) in Theorem~\ref{thm:KMS-full} is obviously satisfied. As for (2), by Lemma~\ref{lem:dimension}(2) we have $\dim_\tau B_\lambda\le|N^{-1}(e^\lambda)|$, but a more careful analysis shows that we can do better than this. Namely, observe that the elements $v_t$ for $t\in\ker N$ become unitary in the GNS-representation defined by $\tau$, hence $\|v_{st}v_t^*-v_s\|_\tau=0$ for every $s$, where $\|\stub\|_\tau$ is defined by~\eqref{eq:semi-tau}. As in~\cite{NS}, define an equivalence relation $\sim_N$ on $S$ by
\[
s\sim_N t\quad\text{if and only if}\quad sa=tb\quad\text{for some}\quad a,b\in\ker N.
\]
Then we see that the right $A_0$-module $\tilde B_\lambda$ generated by representatives of the equivalence classes in $N^{-1}(e^\lambda)$ is dense in $B_\lambda$ with respect to the semi-norm $\lVert\stub\rVert_\tau$. By Lemma~\ref{lem:dimension} we can therefore conclude that
\[
\dim_\tau B_\lambda=\dim_\tau \tilde B_\lambda\le|N^{-1}(e^\lambda)/{\sim_N}|.
\]

Consider the $\zeta$-function of $N$ defined by $\zeta_N(\beta)=\sum_{s\in S/{\sim_N}}N(s)^{-\beta}$. Then by Theorem~\ref{thm:KMS-full} we get that
if $\beta>0$ is such that
\[
\zeta_N(\beta)-1=\smashoperator[r]{\sum_{s\in (S\setminus\ker N)/{\sim_N}}} \ N(s)^{-\beta}<1,
\]
then we have a one-to-one correspondence between the KMS$_\beta$-states on $C^*(S)$ and the tracial states on~$C^*(\ker N)$. By~\cite{NS}*{Theorem~4.5} the same is actually true for all $\beta>0$ such that $\zeta_N(\beta)<\infty$.

Furthermore, for $\beta>0$ such that $\zeta_N(\beta)<\infty$ the KMS$_\beta$-state $\Tr^\F_\tau(e^{-\beta D})^{-1}\Tr^\F_\tau(\Lambda(\stub)e^{-\beta D})$ corresponding to a tracial state $\tau$ on $C^*(\ker N)$ can be explicitly computed as follows. As $C^*(\ker N)=A_c$ is a quotient of $A_0$, we can view $\tau$ as a state on $A_0$. Then $\tau$ is zero on $I$ and a short computation reveals that the spaces $v_sA_0\otimes_{A_0}H_\tau=v_sC^*(\ker N)\otimes_{C^*(\ker N)}H_\tau$ are mutually orthogonal for $s$ lying in different equivalence classes. It follows that if $r_i$ are representatives of the equivalence classes in $N^{-1}(e^\lambda)$ and $\xi_i$ are the images of~$v_{r_i}$ in~$X_\lambda$, then
\[
\sum_i\Tr_\tau^{X_\lambda}(\theta_{\xi,\zeta}\theta_{\xi_i,\xi_i})=\sum_i\tau(\langle\xi_i\langle\xi_i,\zeta\rangle,\xi\rangle)=\tau(\langle\zeta,\xi\rangle)
=\Tr_\tau^{X_\lambda}(\theta_{\xi,\zeta})
\]
for all $\xi,\zeta\in\bigcup_j\xi_j C^*(\ker N)$, hence for all $\xi\in X_\lambda$. We conclude that
\[
\Tr_\tau^{X_\lambda}(\Lambda(a))=\sum_i\Tr_\tau^{X_\lambda}(\Lambda(a)\theta_{\xi_i,\xi_i})=\smashoperator[r]{\sum_{r\in N^{-1}(e^\lambda)/\sim_N}} \ \tau(v_r^*av_r)\quad\text{for all}\quad a\in A_0.
\]
From this we get $\Tr^\F_\tau(e^{-\beta D})=\zeta_N(\beta)$ and, after a short computation, that
\[
\Tr^\F_\tau(\Lambda(v_sv_t^*)e^{-\beta D})=\smashoperator[r]{\sum_{r\in S/{\sim_N}\colon sr\sim_N tr}} \ N(sr)^{-\beta}\tau(v_{q_r}^{\phantom{*}}v_{p_r}^*),
\]
where $p_r,q_r\in\ker N$ are such that $srp_r=trq_r$. This gives exactly the expression for the $\sigma^N$-KMS$_\beta$-states on $C^*(S)$ obtained in~\cite{NS}*{Theorem~4.5}.
\end{example}

\bigskip

\section{K-theory for circle actions}

\subsection{Functoriality problems}\label{ssec:funct}
In this section we begin discussing the problem of relating the K-theory of a C$^*$-algebra to the K-theory of its crystal. But first we want to give a couple of examples to curb our expectations.

\begin{example}
Let us consider again a Pimsner--Toeplitz algebra $A=\TT_Y$ as in Example~\ref{ex:TP}. Consider the gauge action $\gamma$ on $\TT_Y$. In this case we know that, on the one hand, the crystal of $(\TT_Y,\gamma)$ is the coefficient algebra $B$, while on the other hand, the embedding map $B\to\TT_Y$ is a KK-equivalence by~\cite{MR1426840}, so we see that the crystal has the same K-theory as the algebra. At a first glance it seems that this is the perfect situation we hoped for, but a closer look reveals that even in this case there are some problems.

The issue is that the embedding map $B\to\TT_Y$ is given to us by the construction of the algebra~$\TT_Y$, but it has nothing to do with the crystal and in fact cannot be reconstructed from the pair $(\TT_Y,\gamma)$. This is a consequence of the following example given in~\cite{MR4153628}.

Consider the two directed $6$-vertex graphs $E$ and $F$ as in Figure~\ref{fig:ctexmpl-grphs}, that differ only in the range of the edge $e$.
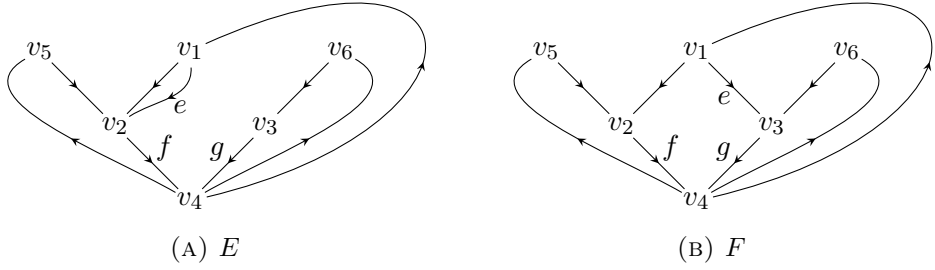
\begin{figure}[ht]
\begin{subfigure}[b]{0.4\linewidth}
\begin{tikzpicture}
    \node[circle, inner sep=0.2pt] (a) at (0,0) {$v_4$};
    \node[circle, inner sep=0.2pt] (b) at (-1,1) {$v_2$};
    \node[circle, inner sep=0.2pt] (c) at (1,1) {$v_3$};
    \node[circle, inner sep=0.2pt] (d) at (-2,2) {$v_5$};
    \node[circle, inner sep=0.2pt] (e) at (0,2) {$v_1$};
    \node[circle, inner sep=0.2pt] (f) at (2,2) {$v_6$};
    \draw[->-] (b)--(a) node [pos=0.5, anchor=south west, inner sep=1pt] {$f$};
    \draw[->-] (c)--(a) node [pos=0.5, anchor=south east, inner sep=1pt] {$g$};
    \draw[->-] (d)--(b);
    \draw[->-] (e) to (b);
    \draw[->-, in=33, out= 273] (e) to node [pos=0.5, anchor=north west, inner sep=1pt] {$e$} (b);
    \draw[->-] (f)--(c);
    \draw[->-, in=210, out=150] (a) to (d);
    \draw[->-] (a)  .. controls +(4.2, 1) and +(4,1.8) .. (e);
    \draw[->-, in=330, out=30] (a) to (f);
\end{tikzpicture}
\caption{$E$}
\end{subfigure}
\begin{subfigure}[b]{0.4\linewidth}
\begin{tikzpicture}
    \node[circle, inner sep=0.2pt] (a) at (0,0) {$v_4$};
    \node[circle, inner sep=0.2pt] (b) at (-1,1) {$v_2$};
    \node[circle, inner sep=0.2pt] (c) at (1,1) {$v_3$};
    \node[circle, inner sep=0.2pt] (d) at (-2,2) {$v_5$};
    \node[circle, inner sep=0.2pt] (e) at (0,2) {$v_1$};
    \node[circle, inner sep=0.2pt] (f) at (2,2) {$v_6$};
    \draw[->-] (b)--(a) node [pos=0.5, anchor=south west, inner sep=1pt] {$f$};
    \draw[->-] (c)--(a) node [pos=0.5, anchor=south east, inner sep=1pt] {$g$};
    \draw[->-] (d)--(b);
    \draw[->-] (e) to (b);
    \draw[->-] (e) to node [pos=0.5, anchor=north east, inner sep=1pt] {$e$} (c);
    \draw[->-] (f)--(c);
    \draw[->-, in=210, out=150] (a) to (d);
    \draw[->-] (a)  .. controls +(4.2, 1) and +(4,1.8) .. (e);
    \draw[->-, in=330, out=30] (a) to (f);
\end{tikzpicture}
\caption{$F$}
\end{subfigure}
\caption{Nonisomorphic graphs with  equivariantly isomorphic C*-algebras.}%Graphs giving counterexample.
\label{fig:ctexmpl-grphs}
\end{figure}
Consider the corresponding Cuntz--Krieger--Toeplitz algebras $\TT C^*(E)$ and $\TT C^*(F)$. Recall that they can be described as Pimsner--Toeplitz algebras as follows. The coefficient algebra is $B=\C^6$, the algebra of functions on the vertices. The C$^*$-correspondences $Y_E$ and $Y_F$ for both graphs are the spaces of functions on the edges, with the  bimodule structure and inner product given by
\[
(a\cdot\xi\cdot b)(h)=a(r(h))b(s(h))\xi(h),\quad \langle \xi,\zeta\rangle(v)=\sum_{h\in s^{-1}(v)} \overline{\xi(h)}\zeta(h)
\]
for all $a,b\in B$ and every vertex $v$ and edge $h$ in the respective graph, where $r$ and $s$ denote the range and the source maps.

For every vertex $v_i$, denote by $q_i\in B$ the delta-function at $v_i$. For every edge $h$, denote by $t_h$ the generator $S_\xi$ of the respective Pimsner--Toeplitz algebra, with $\xi$ equal to the delta-function at~$h$.  As is shown in~\cite{MR4153628}, we have an isomorphism $\pi\colon \TT C^*(E)\to \TT C^*(F)$ such that
\begin{equation}\label{eq:graphs}
q_2\mapsto q_2+t_et_e^*,\quad q_3\mapsto q_3-t_et_e^*,\quad t_f\mapsto t_f+t_gt_et_e^*,\quad t_g\mapsto t_g-t_gt_et_e^*
\end{equation}
and $q_i\mapsto q_i$ for $i\ne2,3$, $t_h\mapsto t_h$ for $h\ne f,g$. The isomorphism is obviously equivariant with respect to the gauge actions, so $(\TT C^*(E),\gamma)\cong (\TT C^*(F),\gamma)$, but the graphs $E$ and $F$ are not isomorphic. Note that this does not contradict our Example~\ref{ex:TP2}, where we showed that from $(\TT_Y,\gamma)$ we can reconstruct the right $B$-module $Y$ and the $\TT_Y$-$B$-correspondence $\F_Y$, but never claimed that we can recover the left action of $B$ on $Y$.

Now, from \eqref{eq:graphs} we see that $\pi\colon \TT C^*(E)\to \TT C^*(F)$ induces the identity map $B\to B$ of the crystals. On the other hand, if we identify the K-theory of both algebras with $K_0(B)=\Z^6$ using the embeddings $B\to \TT C^*(E)$ and $B\to \TT C^*(F)$, then we see that $\pi$ induces a nontrivial map on~$\Z^6$, namely, the one given by the matrix
\[
\begin{pmatrix}
1 & 1& -1 & & & \\
 &  1&  & & & \\
 & &  1& & & \\
 & & & 1& & \\
 & & & & 1& \\
 & & & & & 1\\
\end{pmatrix},
\]
since $t_e^*t_e=q_1$ both in $\TT C^*(E)$ and $\TT C^*(F)$.

We can view $B$ itself as a C$^*$-algebra with trivial dynamics and crystal $B$, so that the embeddings $B\to \TT C^*(E)$ and $B\to \TT C^*(F)$ induce the identity maps on the crystals. The conclusion is that even if we consider only the Pimsner--Toeplitz algebras and gauge actions, there is no natural transformation, in either direction, between the functors
\begin{equation}\label{eq:two-functors}
(A,\sigma)\mapsto K_*(A),\quad (A,\sigma)\mapsto K_*(A_c)
\end{equation}
such that it gives the identity map when $\sigma$ is trivial.
\end{example}

\begin{example}
Consider the CAR-algebra $A=\CAR(H)$ over a separable Hilbert space $H$ as in Example~\ref{ex:CAR}, and let $\gamma$ be the gauge action. As we know, then $A_c\cong\C$.
We have
\[
K_0(A)\cong
\begin{cases}
\Z[1/2],&\text{if $\dim H=\infty$},\\
\Z,&\text{if $\dim H<\infty$},
\end{cases}
\]
and $K_0(A_c)\cong\Z$, so the K-groups of $A$ and $A_c$ are rationally isomorphic.

For finite dimensional Hilbert spaces $H$ we even have isomorphic K-groups, but these isomorphisms cannot be natural. Indeed, if $H'\subset H$ is a subspace of codimension one, then the embedding map $i\colon \CAR(H')\to \CAR(H)$ induces an isomorphism $i_c$ of the corresponding crystals, while the map $i_*\colon\Z\cong K_0(\CAR(H'))\to\Z\cong K_0(\CAR(H))$ is the multiplication by $2$. Therefore there is no natural transformation, in either direction, between the functors~\eqref{eq:two-functors} such that it gives an isomorphism for the algebras $\CAR(H)$ with $\dim H<\infty$. \ee
\end{example}

The conclusions we may draw from these two examples are that we should probably not hope to find a canonical way of relating the K-theory of an algebra to that of its crystal that works in great generality, and that it might be better to consider the rational K-theory.

\subsection{Semi-saturated actions}\label{ssec:semis}
Assume now that we have a circle action $\sigma$ on a separable C$^*$-algebra $A$.
We want to compare $K_*(A)$ with $K_*(A_c)$. We will make two additional assumptions to simplify matters:
\begin{equation}\label{eq:semi-saturated}
A=C^*(A_0,A_1)
\end{equation}
and
\begin{equation}\label{eq:A1-full}
I_{-1}=\overline{A_1^*A_1}=A_0.
\end{equation}

In the terminology of~\cite{MR1276163} the first condition means that we consider a semi-saturated circle action. By~\cite{MR1276163}*{Proposition~4.8} it is equivalent to the condition that $A_n$ is the closure of $A_1^n$ for all $n\ge1$. This implies that $I_1\supset I_2\supset...$, hence
\[
I=I_1\quad\text{and}\quad A_c=A_0/I_1.
\]

Let us introduce the following notation. Assume $B$ and $C$ are C$^*$-algebras and $Y$ is a $B$-$C$ C$^*$-correspondence with the left action of $B$ given by a $*$-homomorphism $B\to\K(Y)$. In this case $Y$ can be viewed as an even Kasparov $B$-$C$-module with zero odd part. We denote the corresponding class by $[Y]\in KK(B,C)$.

By condition~\eqref{eq:A1-full} we therefore have a well-defined class $[A_{-1}]\in KK(A_0,A_0)$. This class defines an endomorphism of $K_*(A_0)=KK_*(\C,A_0)$, $x\mapsto x\otimes_{A_0}[A_{-1}]$, which we denote by $\beta$. This way $K_*(A_0)$ can be viewed as a $\Z[t]$-module, with $t$ acting by $\beta$.

\begin{theorem}\label{thm:K-theory-circle}
Assume $\sigma$ is a semi-saturated circle action on a separable C$^*$-algebra $A$ satisfying condition~\eqref{eq:A1-full}. Assume in addition that
\begin{enumerate}
\item the $\Q[t]$-module $K_*(A_0)\otimes_\Z\Q$ is finitely generated;
\item the action of $t$ on $K_*(A_0)\otimes_\Z\Q$ has zero kernel and no nonzero fixed points.
\end{enumerate}
Then $K_*(A)\otimes_\Z\Q\cong K_*(A_c)\otimes_\Z\Q$.
\end{theorem}

Before we turn to the proof, let us discuss how reasonable it is to expect condition (2) to be satisfied.

As $A_{-1}\cong A_{-1}\otimes_{I_1}A_0$, the class $[A_{-1}]\in KK(A_0,A_0)$ is the composition of $[{}_{A_0}{A_{-1}}_{I_1}]\in KK(A_0,I_1)$ and the class $[\iota]\in KK(I_1,A_0)$ of the embedding $\iota\colon I_1\to A_0$. Since ${}_{A_0}{A_{-1}}_{I_1}$ is an $A_0$-$I_1$-imprimitivity bimodule, we thus see that $\beta$ is the composition of an isomorphism $K_*(A_0)\cong K_*(I_1)$ with  $\iota_*\colon K_*(I_1)\to K_*(A_0)$. Therefore the action of $t$  on $K_*(A_0)\otimes_\Z\Q$ has zero kernel if and only if $\iota_*$ is rationally injective. This is, for example, the case when $I_1$ is unital, as then $I_1$ is a direct summand of $A_0$. Note that if $A_0$ is unital (equivalently, $A$ is unital), then $I_1\cong\K(A_1)$ is unital if and only if $A_1$ is finitely generated as a right $A_0$-module.

Next, it is immediate that $[A_{-1}]^n=[A_{-n}]$. If we denote by $\iota_n$ the embedding map $I_n\to A_0$, then arguing in the same way as in the case $n=1$ above we conclude that
\[
\beta^n(K_*(A_0))=\iota_{n*}(K_*(I_n)),
\]
and therefore
\begin{equation}\label{eq:intersection-zero0}
\bigcap^\infty_{n=0}t^n(K_*(A)\otimes_\Z\Q)=\bigcap^\infty_{n=0}\iota_{n*}(K_*(I_n))\otimes_\Z\Q.
\end{equation}
It follows that the requirement that $t$ has no nonzero fixed points in $K_*(A_0)\otimes_\Z\Q$ is weaker than the condition
\begin{equation}\label{eq:intersection-zero}
\bigcap^\infty_{n=0}\iota_{n*}(K_*(I_n))\otimes_\Z\Q=0.
\end{equation}

\begin{lemma}\label{lem:Lambda-faithful}
For every semi-saturated circle action, the homomorphism $\Lambda\colon A\to\LL(\F)$ is faithful if and only if $\bigcap^\infty_{n=1}I_n=0$.
\end{lemma}

\bp
Observe first that, given $a\in I_n$ for some $n\ge1$, the action of $a$ on $X_n=A_n/\overline{A_n I_1}$ is zero if and only if $a\in I_{n+1}$. Indeed, if $a$ acts by zero, then $aA_n\subset \overline{A_n I_1}$, hence
\[
aA_nA_n^*\subset \overline{A_n I_1A_n^*}= \overline{A_n A_1A_1^*A_n^*}=I_{n+1},
\]
so that $aI_n\subset I_{n+1}$ and therefore $a\in I_{n+1}$. Conversely, if $a\in I_{n+1}$, then
\[
aA_n\subset\overline{A_{n+1}A_{n+1}^*A_n}=\overline{A_{n+1}A_1^*}=\overline{A_nI_1},
\]
so $a$ acts by zero.

It is also clear that $a\in A_0$ acts by zero on $X_0=A_0/I_1$ if and only if $a\in I_1$. These two observations already imply that $\Lambda$ is faithful on $A_0$ if and only if $\bigcap^\infty_{n=1}I_n=0$. But since $\Lambda$ is equivariant with respect to the circle action, it is faithful if and only if its restriction to $A_0$ is faithful.
\ep

The condition $\bigcap_n I_n=0$ does not imply \eqref{eq:intersection-zero} in general. But the lemma above does show that~\eqref{eq:intersection-zero}, and hence also condition (2) in Theorem~\ref{thm:K-theory-circle}, is reasonable to expect to be satisfied as long as $\Lambda$ is faithful, which is in turn a natural condition if we hope to recover information about~$A$ from~$A_c$.

\bp[Proof of Theorem~\ref{thm:K-theory-circle}]
Consider the $A_0$-$A_0$ C$^*$-correspondence $Y=A_{-1}$. Then we can identify $A$ with the Cuntz--Pimsner algebra $\tilde\OO_Y$ from~\cite{MR1426840}, where tilde is to remind that the right C$^*$-Hilbert $A_0$-module $A_{-1}$ is not full. For this C$^*$-algebra we still have a six-term exact sequence similar to~\cite{MR1426840}*{Theorem~4.9},
\begin{equation}\label{eq:6-term1}
\begin{tikzcd}
K_0(A_0) \ar[r, "\id-\beta"] & K_0(A_0)\ar[r] & K_0(A)\ar[d]\\
K_1(A)\ar[u]  & K_1(A_0) \ar[l] & K_1(A_0) \ar[l, "\id-\beta"].
\end{tikzcd}
\end{equation}
This can be proved in the same way as in the case of full modules. Alternatively, we can stabilize $A$, write it  using \cite{MR1276163}*{Theorem~4.21} as the covariance algebra of a partial automorphism $\theta\colon A_0\to I_1$ of $A_0$, and then apply \cite{MR1276163}*{Theorem~7.1}.

On the other hand, the six-term exact sequence in K-theory defined by the short exact sequence $0\to I_1\to A_0\to A_c\to0$ can be written as
\begin{equation}\label{eq:6-term2}
\begin{tikzcd}
K_0(A_0) \ar[r, "\beta"] & K_0(A_0)\ar[r] & K_0(A_c)\ar[d]\\
K_1(A_c)\ar[u]  & K_1(A_0) \ar[l] & K_1(A_0) \ar[l, "\beta"].
\end{tikzcd}
\end{equation}

Exact sequences \eqref{eq:6-term1} and \eqref{eq:6-term2} and assumption (2) imply that
\[
K_*(A)\otimes_\Z\Q\cong M/(1-t)M\quad\text{and}\quad K_*(A_c)\otimes_\Z\Q\cong M/tM,
\]
where $M=K_*(A_0)\otimes_\Z\Q$. By assumption (1), the $\Q[t]$-module $M$ is finitely generated, hence it is a finite direct sum of modules of the form $N=\Q[t]/(f)$, while assumption (2) means that $f(0),f(1)\ne0$ unless $f=0$. But then $N/(1-t)N\cong N/tN\cong\Q$ if $f=0$ and $N/(1-t)N=N/tN=0$ otherwise, completing the proof of the theorem.
\ep

\begin{example}
Consider once again a Pimsner--Toeplitz algebra $A=\TT_Y$ as in Example~\ref{ex:TP} together with the gauge action $\gamma$.
We assume that the coefficient algebra $B$ is separable and $Y$ is countably generated as a right C$^*$-Hilbert $B$-module.
In this case we know that $K_*(A)\cong K_*(B)=K_*(A_c)$, but let us see what the assumptions of Theorem~\ref{thm:K-theory-circle} mean in this case.

The gauge action is always semi-saturated. Condition~\eqref{eq:A1-full} is satisfied if $Y$ is full as a right C$^*$-Hilbert $B$-module.
The embedding $B\to A_0$ induces a homomorphism $K_*(B)\to K_*(A_0)$, which in turn extends to a $\Z[t]$-module map $K_*(B)[t]\to K_*(A_0)$. It is not very difficult to see that this is actually an isomorphism. Therefore condition (2) in Theorem~\ref{thm:K-theory-circle} is always satisfied, while (1) is satisfied if and only if the abelian group $K_*(B)$ has finite rank.
\end{example}

\begin{example}\label{ex:dynam}
Assume $X$ is a separable totally disconnected locally compact space, $T\colon X\to X$ is a homeomorphism, and $Y\subset X$ is a compact open subset such that
\[
TY\subset Y\qquad\text{and}\qquad \bigcup_{n\in\Z}T^nY=X.
\]
Let $\alpha_T$ be the automorphism of $C_0(X)$ defined by $T$, $\alpha_T(f)(x)=f(T^{-1}x)$, and let
\[
A=\un_Y(C_0(X)\rtimes_{\alpha_T}\Z)\un_Y.
\]
Consider the dynamics on $A$ defined by the dual action of $\T$:
\[
\sigma_t(fu^n)=e^{2\pi in}fu^n,
\]
where $u\in M(C_0(X)\rtimes_{\alpha_T}\Z)$ is the unitary implementing $\alpha_T$.

Equivalently, the algebra $A$ is the C$^*$-algebra of the reduction $\G$ of the transformation groupoid $\Z\ltimes X$ by $Y$, so
\[
\G=\{(n,y)\in \Z\times Y\colon T^ny\in Y\},
\]
and the dynamics is defined by the $1$-cocycle $c(n,y)=2\pi n$. The boundary set of this cocycle is $Z=Y\setminus TY$. Hence, by Proposition~\ref{prop:groupoid-crystal}, the crystal of $(A,\sigma)$ is
\[
A_c=C(Z).
\]
Conditions~\eqref{eq:semi-saturated} and~\eqref{eq:A1-full} are easily seen to be satisfied. It is also easy to see that
$I_n=C(Y)\un_{T^n Y}$, so that $\cap_{n\ge1}I_n=0$ if and only if the set
\[
Y_\infty=\bigcap_{n\ge0} T^nY
\]
has empty interior, or equivalently, $\cup_{n\ge0}T^nZ$ is dense in $Y$. Let us assume from now on that this condition is satisfied.
Denote by $T_\infty$ the homeomorphism $T|_{Y_\infty}$.

Let us try to understand when the assumptions of Theorem~\ref{thm:K-theory-circle} are satisfied. We have $A_0=C(Y)$, $K^0(Y)=K_0(C(Y))\cong C(Y;\Z)$, and the $\Z[t]$-module structure is defined so that $t$ acts by the restriction of $\alpha_T$ to $C(Y;\Z)\subset C(Y)$. Then $K^0(Y)\otimes_\Z\Q\cong C(Y;\Q)$, where $\Q$ is considered with the discrete topology. It is clear that the action of $t$ on $C(Y;\Q)$ has zero kernel and no fixed points. So the only nontrivial condition is the finite generation of $C(Y;\Q)$ as a $\Q[t]$-module.

We have $K^0(Y_\infty)=K_0(C(Y_\infty))\cong C(Y_\infty;\Q)$, and $C(Y_\infty;\Q)$ has a $\Q[t]$-module structure defined by $tf=f\circ T_\infty^{-1}$. As $T_\infty$ is invertible, we actually have a $\Q[t,t^{-1}]$-structure. We can identify $K^0(Y\setminus Y_\infty)=K_0(C_0(Y\setminus Y_\infty))$ with $K^0(Z)[t]$ as a $\Z[t]$-module. Then the short exact sequence in K-theory corresponding to the extension $0\to C_0(Y\setminus Y_\infty)\to C(Y)\to C(Y_\infty)\to0$ gives the following exact sequence of $\Z[t]$-modules:
\begin{equation}\label{eq:function-modules}
0\to C(Z;\Z)[t]\to C(Y;\Z)\to C(Y_\infty;\Z)\to0.
\end{equation}
(The exactness is of course also easy to see directly.) This implies that the $\Q[t]$-module $C(Y;\Q)$ is finitely generated if and only if the module $C(Y_\infty;\Q)$ is finitely generated and the space $C(Z;\Q)$ is finite dimensional. The last condition is equivalent to finiteness of $Z$.

We claim that the $\Q[t]$-module $C(Y_\infty;\Q)$ is finitely generated if and only if $Y_\infty$ is finite. The ``if'' direction is obvious. Assume now that the module $C(Y_\infty;\Q)$ is finitely generated. Then it is a finite direct sum of cyclic modules $\Q[t]/(f)$. Since the action of $t$ on $C(Y_\infty;\Q)$ is invertible, we cannot have $f=0$. This implies that $C(Y_\infty;\Q)$ is actually a finite dimensional space over $\Q$, hence~$Y_\infty$ is finite.

To summarize, if $\cup_{n\ge0}T^nZ$ is dense in $Y$, then the conditions of  Theorem~\ref{thm:K-theory-circle} are satisfied if and only if the sets $Z$ and $Y_\infty$ are finite, and then the theorem states that $K_0(A)\otimes_\Z\Q\cong K_0(A_c)\otimes_\Z\Q= K^0(Z)\otimes_\Z\Q\cong\Q^{|Z|}$. In fact, with a bit more effort one can show that then
\[
K_0(A)\cong K_0(A_c)\cong\Z^{|Z|}.
\]

Without any assumptions on $Z$ and $Y_\infty$, the exact sequence~\eqref{eq:function-modules} gives rise to the exact sequence
\[
K^0(Z)\to K^0(Y)/(1-t)K^0(Y)\to K^0(Y_\infty)/(1-t)K^0(Y_\infty)\to0.
\]
From \eqref{eq:6-term1} we also have $K_0(A)\cong K^0(Y)/(1-t)K^0(Y)$. (This can also be deduced from the usual Pimsner--Voiculescu sequence for $C_0(X)\rtimes\Z$ by observing that $C(Y;\Z)\cap (1-t)C_0(X;\Z)=(1-t)C(Y;\Z)$.)
We also have $[\un_Z]=(1-t)[\un_Y]$ in $K^0(Y)$ and $K^0(Y_\infty)/(1-t)K^0(Y_\infty)\cong K_0(C(Y_\infty)\rtimes_{\alpha_{T_\infty}}\Z)$. Therefore we get the exact sequence
\begin{equation}\label{eq:exact-cross}
K^0(Z)/\Z[\un_Z]\to K_0(A)\to K_0(C(Y_\infty)\rtimes_{\alpha_{T_\infty}}\Z)\to0.
\end{equation}

Using this let us show that when $Z$ consists of one point and $\cup_{n\ge0}T^nZ$ is dense in $Y$, but there are no restrictions on $Y_\infty$, the K-theory of our one-point crystal does not say anything interesting about the K-theory of $A$. For this, take an arbitrary homeomorphism $S$ of a totally disconnected separable compact space $W$ having a dense forward orbit $\{S^nx_0\}_{n\ge0}$. Define sets
\[
X=\Z\sqcup W,\qquad Y=\Z_+\sqcup W.
\]
Let $T\colon X\to X$ be the bijection such that $T(n)=n+1$ and $T|_W=S$. Consider a metric $d$ on $W$ defining the topology on $W$. To extend this metric to $X$, consider the function $f$ on $\Z$ defined by
\[
f(n)=
\begin{dcases*}
-n+1,& if $n\le0$\\
\frac{1}{n},& if $n\ge1$,
\end{dcases*}
\]
and then put
\begin{align*}
d(m,n)&=|f(m)-f(n)|+d(T^mx_0,T^nx_0),\\
d(n,x)&=f(n)+d(T^nx_0,x)\quad\text{for}\quad x\in W.
\end{align*}
Then $Y$ is a compact open subset of $X$, $T$ is a homeomorphism, the relative topology on $W$ is the original one, $d(T^n(0),T^nx_0)\to0$ as $n\to+\infty$ and $T^n(0)\to\infty$ as $n\to-\infty$. It is not difficult to see that these properties uniquely characterize the metrizable topology on $X$.

By construction we have $Z=\{0\}$ and $(Y_\infty,T_\infty)=(W,S)$. The density of the forward orbit of~$x_0$ in~$W$ and the property $d(T^n(0),T^nx_0)\to0$ as $n\to+\infty$ imply that $\cup_{n\ge0}T^nZ$ is dense in $Y$. Note that the Fock space in this case can be identified with $\ell^2(\Z_+)$ and $A$ can be identified with the C$^*$-subalgebra of $B(\ell^2(\Z_+))$ generated by the shift operator and by the operators of multiplication by $f|_{\Z_+}$ for $f\in C(Y)$. From~\eqref{eq:exact-cross} we get
\[
K_0(A)\cong K_0(C(W)\rtimes_{\alpha_S}\Z)\qquad\text{and}\qquad K_0(A_c)\cong K^0(Z)\cong\Z.
\]

In this construction we can in particular start with any minimal homeomorphism $S$ of a Cantor set and take any point $x_0$. As a concrete example consider an interval exchange transformation~$S$ defined by a permutation $\sigma\in S_n$ ($n\ge2$) and a probability distribution $(\lambda_1,\dots,\lambda_n)$. It is known that if $\sigma$ is irreducible, then $S$ is minimal for generic $(\lambda_1,\dots,\lambda_n)$. Assuming that $S$ is indeed minimal, by~\cite{MR978619}*{Theorem~2.1} we then get
\[
K_0(C(W)\rtimes_{\alpha_S}\Z)\cong \Z^n.
\]
As a side note, although $K^0(W)\cong C(W;\Z)$ cannot be finitely generated as a $\Z[t]$-module for infinite~$W$ by the earlier discussion, in the present case it is finitely generated as a $\Z[t,t^{-1}]$-module by the proof of~\cite{MR978619}*{Lemma~2.3}.
\end{example}

\bigskip

\section{Crystals and K-theory of inverse semigroups}\label{sec:inverse}

\subsection{Scales on inverse semigroups}

Let $I$ be an inverse semigroup. Denote by $I^\times\subset I$ the subset of nonzero elements.
By a \emph{scale} on~$I$ we mean a map
$N\colon I^\times\to\R^\times_+$ such that
\[
N(gh)=N(g)N(h)\quad\text{whenever}\quad gh\ne0.
\]

\begin{example}\label{ex:scale}
Assume $S$ is a left cancellative monoid and $N\colon S\to\R^\times_+$ is a homomorphism. Consider the (left) inverse hull $I_\ell(S)$ of $S$, that is, the inverse semigroup of partial bijections on $S$ generated by the left translations $S\to sS$, $t\mapsto st$ ($s\in S$). We view $S$ as a subset of $I_\ell(S)$ and denote by $s^{-1}\in I_\ell(S)$ the inverse of $s$ in $I_\ell(S)$, that is, the partial bijection $sS\to S$, $st\mapsto t$.

We claim that~$N$ extends uniquely to a scale on $I_\ell(S)$. For this it suffices to show that if $s_n^{-1}t_n\dots s_1^{-1}t_1=\id_X$ in $I_\ell(S)$ for some $s_i,t_i\in S$ and $X\subset S$, $X\ne\emptyset$, then
\[
N(s_n)^{-1}N(t_n)\dots N(s_1)^{-1}N(t_1)=1.
\]
Take any $v_0\in X$. Define by induction elements $v_1,\dots,v_n\in S$ by letting $v_i=s_i^{-1}t_i v_{i-1}$, so that $v_n=v_0$. As $s_iv_i=t_i v_{i-1}$, we have $N(v_i)=N(s_i)^{-1}N(t_i)N(v_{i-1})$ for all $i=1,\dots,n$. It follows that $N(v_0)=N(v_n)=N(s_n)^{-1}N(t_n)\dots N(s_1)^{-1}N(t_1)N(v_0)$, proving our claim.\ee
\end{example}

Consider the modified Paterson groupoid $\G_0(I)$ of $I$~\cite{MR1724106}, where \emph{modified} means that we exclude the semi-character $1$ from the unit space when $0\in I$. Namely,
let $E$ be the abelian semigroup of idempotents in~$I$, $\Omega=\hat E$ the collection of nonzero semigroup homomorphisms $\chi\colon E\to\{0,1\}$, where $\{0,1\}$ is considered as a semigroup under multiplication, such that $\chi(0)=0$ in case $0\in I$. Then
\[
\G_0(I)=\Sigma/{\sim},\quad \text{where}\quad \Sigma = \{(g,\chi) \in I\times\Omega \colon \chi(g^{-1}g)=1\},
\]
and the equivalence relation $\sim$ is defined by declaring $(g_1,\chi_1)$ and $(g_2,\chi_2)$ to be equivalent if and only if
\[
 \chi_1=\chi_2 \text{ and there exists $p\in E$ such that $g_1p=g_2p$ and $\chi_1(p)=1$}.
\]
We denote by $[g,\chi]$ the class of $(g,\chi)\in\Sigma$ in $\G_0(I)$.

The product is defined by
\[
[g,\chi]\,[h,\psi] = [gh,\psi]\quad\text{if}\quad \chi=\psi(h^{-1} \stub h).
\]
The unit space $\G_0(I)^{(0)}$ of the groupoid $\G_0(I)$ can be identified with $\Omega$ via the map $\Omega\to\G_0(I)$, $\chi\mapsto[p,\chi]$, where $p\in E$ is any idempotent satisfying $\chi(p)=1$. The source and range maps are then given by
\[
s([g,\chi])=\chi,\qquad r([g,\chi])=g.\chi=\chi(g^{-1} \stub g),
\]
while the inverse is given by $[g,\chi]^{-1}=[g^{-1},g.\chi]$.

We consider the topology of pointwise convergence on $\Omega$. Equivalently, for every $p\in E^\times=E\setminus\{0\}$,
consider the set
\begin{equation}\label{eq:Omegap}
\Omega_p=\{\chi\in\Omega\colon \chi(p)=1\}.
\end{equation}
Then a basis of the topology on $\Omega$ is formed by the sets $\Omega_p\setminus(\Omega_{p_1}\cup\dots\cup\Omega_{p_n})$. The space $\Omega$ is locally compact and Hausdorff.

For $g\in I^\times$ and an open subset $U\subset\Omega_{g^{-1}g}$, define
\[
D(g,U)= \{[g,\chi]\in\G_0(I)\colon \chi \in U\}.
\]
Then the topology on $\G_0(I)$ is defined by taking as a basis the sets $D(g,U)$. This turns $\G_0(I)$ into a locally compact, but not necessarily Hausdorff, \'{e}tale groupoid.

We then have a canonical isomorphism $C^*(\G_0(I))\cong C^*_0(I)$, where the subscript $0$ indicates that if $I$ contains the zero element, then we consider only those representations of $I$ by partial isometries that map $0\in I$ into $0$.

\smallskip

Assuming that we are given a scale $N$, define a $1$-cocycle $c_N\colon\G_0(I)\to\R_d$,
\[
c_N([g,\chi])=\log N(g).
\]
Consider the corresponding dynamics $\sigma^N$ on $C^*(\G_0(I))=C^*_0(I)$ as described in Section~\ref{sec:crystal-gropd}; thus, if~$v_g$ ($g\in I^\times$) are the generators of $C^*_0(I)$, then $\sigma^N_t(v_g)=N(g)^{it}v_g$. Consider also the corresponding boundary set $Z\subset\Omega$ defined by~\eqref{eq:boundary}. Define a (possibly empty) set
\[
E_c^\times=\{p\in E^\times\colon \text{if $p=g^{-1}g$, then $N(g)\ge1$}\}.
\]

For every $p\in E^\times$, denote by $\chi_p$ the semi-character defined by
\begin{equation}\label{eq:chi_p}
\chi_p(q)=\begin{cases}
            1, & \text{if $p\le q$}, \\
            0, & \text{otherwise},
          \end{cases}
\end{equation}
where $p\le q$ means that $p=pq=qp$.

\begin{lemma}\label{lem:inverse-boundary-set}
We have
\[
Z=\Omega\setminus\bigcup_{p\in E^\times\setminus E_c^\times}\Omega_p=\overline{\{\chi_p\colon p\in E_c^\times\}}.
\]
In particular, $Z=\emptyset$ if and only if $E_c^\times=\emptyset$.
\end{lemma}

\bp
By definition, $\chi\notin Z$ if and only if there exists $[g,\chi]\in\G(I)$ such that $N(g)<1$. That is, if and only if there exist $p\in E^\times$ and $g\in I$ such that $g^{-1}g=p$, $N(g)<1$ and $\chi(p)=1$. In other words, if and only if there exists $p\in E^\times\setminus E_c^\times$ such that $\chi(p)=1$. This proves the first equality in the formulation.

For the second equality, take $p\in E_c^\times$. If $\chi_p\notin Z$, then $\chi_p\in\Omega_q$ for some $q\notin E_c^\times$. This means that $p\le q=g^{-1}g$ for some $g\in I$ with $N(g)<1$. But then for $h=gp$ we have $h^{-1}h=p$ and $N(h)=N(g)<1$, which contradicts the assumption $p\in E_c^\times$. Therefore $\chi_p\in Z$. It follows that
\[
\overline{\{\chi_p\colon p\in E_c^\times\}}\subset Z.
\]

In order to prove the opposite inclusion, take $\chi\in Z$ and a neighbourhood $U$ of $\chi$. We have
\[
\chi\in \Omega_p\setminus(\Omega_{p_1}\cup\dots\cup\Omega_{p_n})\subset U
\]
for some $p,p_1,\dots,p_n\in E^\times$. But then we must have $p\in E_c^\times$ by definition of $Z$ and $E_c^\times$. As $\chi_p\in U$, this completes the proof.
\ep

Next, define
\[
I_c^\times=\{g\in I^\times\colon g^{-1}g\in E_c^\times, g g^{-1}\!\in E_c^\times\}.
\]
By the definition of $E_c^\times$ it is clear that $N(g)=1$ for all $g\in I_c^\times$.

\begin{lemma}
Assuming $E_c^\times\ne\emptyset$ and adding $0$ to $I_c^\times$ if necessary, we get an inverse semigroup~$I_c$ with product
\[
g\cdot h=\begin{cases}
           gh, & \text{if $gh\in I_c^\times$}, \\
           0, & \text{otherwise}.
         \end{cases}
\]
\end{lemma}

If $E_c^\times=\emptyset$, we let $I_c=\{0\}$.

\bp
We start with two observations. First, the subset $E_c^\times\subset E$ is upward hereditary, meaning that if $0\ne p\le q$ and $p\in E_c^\times$, then $q\in E_c^\times$. This was essentially proved in Lemma~\ref{lem:inverse-boundary-set}. Second, the set~$I_c^\times$ can equivalently be described as
\begin{equation}\label{eq:Ic}
I_c^\times=\{g\in I^\times\colon N(g)=1, g^{-1}g\in E_c^\times\}=\{g\in I^\times\colon N(g)=1,  gg^{-1}\in E_c^\times\},
\end{equation}
since if $N(g)=1$ and $gg^{-1}\notin E^\times_c$, so that $gg^{-1}=h^{-1}h$ for some $h$ with $N(h)<1$, then $g^{-1}g=(hg)^{-1}hg$ and $N(hg)<1$, hence $g^{-1}g\notin E^\times_c$.

Now, the only nontrivial condition left to prove is associativity of the above product on $I_c$.
As $I$ associative, it suffices to show that if $g,h,k\in I_c^\times$ are such that $gh\in I_c^\times$ and $ghk\in I_c^\times$, then also $hk\in I_c^\times$. By the second observation, for this it is enough to show that $(hk)^{-1}hk\in E_c^\times$. But as
\[
(hk)^{-1}hk\ge (ghk)^{-1}ghk\in E_c^\times,
\]
this is true by the first observation.
\ep

\begin{proposition}\label{prop:crystal-of-inv-semi-and-grpd}
For any inverse semigroup $I$ with a scale $N$ and the corresponding boundary set $Z\subset\G_0(I)^{(0)}$, we have a canonical groupoid isomorphism $\G_0(I)|_Z\cong \G_0(I_c)$.
\end{proposition}

\bp
Since $E_c^\times$ is an upward hereditary subset of $E$, every semi-character $\chi$ on $E_c$, satisfying $\chi(0)=0$ if $0\in I_c$, extends to a semi-character on $E$ by letting $\chi(q)=0$ for $q\notin E_c$. This allows us to view $\G_0(I_c)^{(0)}$ as a closed subset of $\G_0(I)^{(0)}$. The second equality in Lemma~\ref{lem:inverse-boundary-set} implies that this subset is exactly the boundary set~$Z$. It is then easy to see that we have a well-defined homomorphism $\pi\colon\G_0(I_c)\to \G_0(I)|_Z$ that maps the class of a pair $(g,\chi)$ in $\G_0(I_c)$ into the class of this pair in $\G_0(I)$. We need to check that this homomorphism is bijective.

Take $[g,\chi]\in \G_0(I)|_Z$. As $\G_0(I)|_Z\subset c_N^{-1}(0)$, we have $N(g)=1$. As $\chi(g^{-1}g)=1$ and $\chi\in Z=\G_0(I_c)^{(0)}$, we then have $g\in I_c^\times$ by~\eqref{eq:Ic}. Therefore the pair $(g,\chi)$ defines an element of $\G_0(I_c)$, so $\pi$ is surjective.

For the injectivity, assume pairs $(g,\chi)$ and $(h,\chi)$ define the same element of $\G_0(I)|_Z$, that is, there exist $p\in E$ such that $\chi(p)=1$ and $gp=hp$. As $\chi\in Z=\G_0(I_c)^{(0)}$, we have $p\in E_c^\times$. Then, again by~\eqref{eq:Ic}, we have  $gp\in I^\times_c$, since $\chi(pg^{-1}gp)=\chi(p)\chi(g^{-1}g)=1$. Similarly, $hp\in I^\times_c$. It follows that $g\cdot p=gp=hp=h\cdot p$ in $I_c$. Therefore $(g,\chi)$ and $(h,\chi)$ define the same element of $\G_0(I_c)$, so~$\pi$ is injective.
\ep

By applying Proposition~\ref{prop:groupoid-crystal} we get the following corollary.

\begin{corollary}\label{cor:inverse-crystal}
The crystal of $(C^*_0(I),\sigma^N)$ is canonically isomorphic to $C^*_0(I_c)$.
\end{corollary}

This motivates the following definition.

\begin{definition}
Given an inverse semigroup $I$ and a scale $N$ on it, we call the inverse semigroup~$I_c$ the \emph{crystal} of $(I,N)$.
\end{definition}

\begin{example}\label{ex:LCM2}
As in Example~\ref{ex:LCM}, consider a right LCM semigroup $S$ and a semigroup homomorphism $N\colon S\to [1,+\infty)$. As was shown in Example~\ref{ex:scale}, it extends to a scale $N_I$ on the inverse hull~$I_\ell(S)$ of~$S$. Assume~$N$ satisfies condition~\eqref{eq:scale-condition}. Let us show that then the crystal of $(I_\ell(S),N_I)$ is $I_\ell(\ker N)$.

By the LCM condition, every element of $I_\ell(S)$ can be written as $s^{-1}t$ for some $s,t\in S$. It follows that the idempotents of $I_\ell(S)$ are the elements $p_{sS}=ss^{-1}$, which are in a one-to-one correspondence with the right principal ideals $sS\subset S$, and possibly $0$. By definition it is immediate that if $p_{sS}\in E^\times_c$, then $s\in\ker N$.

Conversely, take $s\in\ker N$ and assume $p_{sS}=(x^{-1}y)^{-1}(x^{-1}y)$ for some $x,y\in S$. This means that $sS=y^{-1}(xS\cap yS)$, equivalently,
\[
ysS=xS\cap yS.
\]
It follows that $ys=xt$ for some $t\in S$, hence $N(y)=N(ys)=N(xt)\ge N(x)$. This proves that $p_{sS}\in E^\times_c$. We see also that if $N(x)=N(y)$ in the above argument, then $t\in\ker N$ and we have
\[
x^{-1}y=x^{-1}(ys)(ys)^{-1}y=x^{-1}(xt)(ys)^{-1}y=t^{-1}s.
\]
Therefore $E_c^\times=\{p_{sS}:s\in\ker N\}$ and $\ker N\subset I_\ell(S)_c^\times\subset(\ker N)^{-1}(\ker N)\subset I_\ell(S)$.

So far we haven't used condition~\eqref{eq:scale-condition}. This condition guarantees that $\ker N$ is a right LCM semigroup and the embedding $\ker N\hookrightarrow I_\ell(S)$ extends to an embedding $I_\ell(\ker N)\hookrightarrow I_\ell(S)$. It follows that the crystal of $(I_\ell(S),N_I)$ is indeed $I_\ell(\ker N)$, which is consistent with the fact that the crystal of $(C^*(S),\sigma^N)$ is $C^*(\ker N)$ that we already established in Example~\ref{ex:LCM}.
\end{example}

\subsection{K-groups of inverse semigroup algebras}\label{sec:k-grps-inv-semigrps}
Given a scale $N$ on an inverse semigroup $I$, we know by Corollary~\ref{cor:inverse-crystal} that the crystal of $(C^*_0(I),\sigma^N)$ is again the C$^*$-algebra of an inverse semigroup. This allows us to apply recent results of Miller~\cite{miller-thesis} to compare the K-groups and obtain the following result.

\begin{theorem}\label{thm:k-theory-isom-inv-semigrp}
Assume $I$ is a countable inverse semigroup with a scale $N\colon I^\times\to\R^\times_+$. Consider the crystal $I_c$ of $(I,N)$ and assume that the following conditions are satisfied:
\begin{enumerate}
  \item the groupoid $\G_0(I)$ is Hausdorff; \label{it:hausdorff-property-of-universal-groupoid}
  \item the groupoids $\G_0(I)$, $\G_0(I_c)$  and the groups $\Gamma_p=\{g\in I\colon g^{-1}g=gg^{-1}=p\}$ ($p\in E^\times$) satisfy the Baum--Connes conjecture; \label{it:baum-connes-for-groupoids-and-stab-grp}
  \item every nonzero idempotent $p\in E^\times$ is equivalent to some $q\in E_c^\times$, that is, there is $g\in I$ such that $g^{-1}g=p$ and $gg^{-1}\in E_c^\times$.
  \label{it:transversality}
\end{enumerate}
Then $C^*_0(I)$ and the crystal $C^*_0(I_c)$ of $(C^*_0(I),\sigma^N)$ have isomorphic K-groups.
\end{theorem}

To be precise, by the Baum--Connes conjecture for a locally compact Hausdorff étale groupoid~$\HH$ we mean that the KK-morphism $P \rtimes \HH \to C^*(\HH)$ coming from a $\KK^\HH$-morphism $P \to C_0(\HH^{(0)})$ given by the categorical formulation of the Baum--Connes conjecture for groupoids~\cite{arXiv:2202.08067}, induces an isomorphism $K_*(P \rtimes\HH) \to K_*(C^*(\HH))$. See~\cite{MR2104446} for a comparison with the traditional approach to the Baum--Connes conjecture for groups. Note also that since the Baum--Connes conjecture with coefficients passes to closed subgroupoids, see~\cite{arXiv:2202.08067}*{Theorem~4.11}, instead of (\ref{it:baum-connes-for-groupoids-and-stab-grp}) we can assume that $\G_0(I)$ satisfies the Baum--Connes conjecture with coefficients.

In practice, one often proves that the morphism $P \to C_0(\HH^{(0)})$ is a $\KK^\HH$-equivalence.
This implies $K_*(C^*(\HH)) \cong K_*(C^*_r(\HH))$. For example, the class of inverse semigroups $I$ such that $\G_0(I)$ is amenable, studied in~\cite{MR3614034}, gives such $\HH=\G_0(I)$ by~\cite{MR1703305}, see also~\citelist{\cite{MR4448401}\cite{arXiv:2202.08067}}. Note that in this case the groups $\Gamma_p$ are also amenable, since they coincide with the isotropy groups $\G_0(I)^{\chi_p}_{\chi_p}$ of the semi-characters $\chi_p$ defined by~\eqref{eq:chi_p}, hence they satisfy the Baum--Connes conjecture by~\cite{MR1821144}. Similarly, as~$\cG_0(I_c)$ is a closed subgroupoid of $\G_0(I)$  by Proposition~\ref{prop:crystal-of-inv-semi-and-grpd}, amenability of $\cG_0(I)$ implies that of~$\cG_0(I_c)$ by~\cite{MR1799683}*{Proposition 5.1.1}. Therefore if $\G_0(I)$ is Hausdorff and amenable, then we only need to check condition (3) in Theorem~\ref{thm:k-theory-isom-inv-semigrp}.

\begin{proof}[Proof of Theorem~\ref{thm:k-theory-isom-inv-semigrp}]
Assume first that the stabilizer groups $\Gamma_p$ are torsion-free. Following~\cite{arXiv:2401.17240}, we approximate the groupoid $\cG_0(I)$ by a discrete groupoid $\cG_d(I)$, as follows. The object set of $\cG_d(I)$ is $E^\times$ as a set. Given two elements $p,q \in E^\times$, the morphisms from $p$ to $q$ are the elements $g \in I^\times$ such that $g^{-1} g = p$ and $gg^{-1} = q$. Then, as explained in~\cite{arXiv:2401.17240}*{Section 7}, conditions~(\ref{it:hausdorff-property-of-universal-groupoid}) and~(\ref{it:baum-connes-for-groupoids-and-stab-grp}) for~$\cG_0(I)$ imply that there is an isomorphism $K_*(C^*(\G_d(I)))\cong K_*(C^*(\G_0(I)))$. For the same reason we have $K_*(C^*(\G_d(I_c)))\cong K_*(C^*(\G_0(I_c)))$. On the other hand, condition (\ref{it:transversality}) implies that the groupoids $\G_d(I_c)$ and $\G_d(I)$ are Morita equivalent. It follows that $K_*(C^*(\G_d(I_c)))\cong K_*(C^*(\G_d(I)))$.

When the stabilizer groups $\Gamma_p$ have torsion elements, the proof goes along similar lines using an appropriate generalization of the Bredon homology to groupoids, as sketched in~\cite{miller-thesis}. Full details will appear in a forthcoming joint work of Miller, the third-named author and others.
\end{proof}

\begin{example}
As in Example~\ref{ex:LCM}, consider a right LCM semigroup $S$ and a semigroup homomorphism $N\colon S\to [1,+\infty)$. We assume that $S$ is countable and  $N$ satisfies condition~\eqref{eq:scale-condition}. By Example~\ref{ex:LCM2} we know that the crystal of $(I_\ell(S),N_I)$ is $I_\ell(\ker N)$. By the LCM condition all nonzero idempotents in $I_\ell(S)$ are equivalent, so condition (3) in Theorem~\ref{thm:k-theory-isom-inv-semigrp} is satisfied. As for the other two conditions, one can show that they are satisfied if $S$ embeds into a discrete group satisfying the Baum--Connes conjecture with coefficients, and in fact in this case we have, by a variant of \cite{MR4381183}*{Corollary~1.4}, that
\[
K_*(C^*_0(I_\ell(S)))\cong K_*(C^*(S^*))\cong K_*(C^*_0(I_\ell(\ker N))),
\]
where $S^*\subset S$ denotes the group of units.

As a concrete example, we can take the $ax+b$ semigroup $S=\OO_K\rtimes\OO^\times_K$, where $K$ is a number field of class number one and $\OO_K\subset K$ is the ring of algebraic integers, with the scale defined by $N(b,a)=|\OO_K/(a)|$. More generally, if $K$ is not of class number one, then $S=\OO_K\rtimes\OO^\times_K$ is not LCM, but
\[
C^*(\G_0(I_\ell(S)))\cong C^*_0(I_\ell(S))\cong C^*_r(S)
\]
by \cite{MR2900468}*{Lemma~2.30}, \cite{MR3200323}*{Theorem~3.2.14} and amenability of $K\rtimes K^\times$, and Theorem~\ref{thm:k-theory-isom-inv-semigrp} still applies by results of~\cite{MR3323201}*{Section~8}. The K-theory of $C^*_r(S)$ in this case has been already computed in~\cite{MR3323201}*{Theorem~8.2.1}. As was discussed in the introduction, the result is one of the motivations for the present work.

A similar analysis applies to the semigroups $S=\OO_K\rtimes\OO^\times_{K,+}$, where $K$ is a totally real field and $\OO^\times_{K,+}\subset \OO^\times_K$ is the submonoid of totally positive elements. In this case the stabilizer groups $\Gamma_p$ are torsion-free.
\end{example}

\bigskip

\end{document}